\documentclass[11pt]{amsart}
\usepackage[top=3cm,bottom=2cm,left=3cm,right=3cm,marginparwidth=1.75cm]{geometry}
\usepackage[format=plain, font=footnotesize]{caption}
\usepackage{amsmath, amsthm, nccmath, amssymb, mathtools, hyperref}
\hypersetup{
	colorlinks=true,       
	linkcolor=orange,        
	citecolor=purple,        
	filecolor=magenta,     
	urlcolor=blue
}
\usepackage[shortlabels]{enumitem}
\usepackage{graphicx}
\usepackage{xcolor}
\usepackage{listings}
\usepackage{url}
\usepackage{cleveref}
\usepackage{systeme}
\usepackage{array}
\usepackage{bigints}
\usepackage{comment}
\usepackage{tikz-cd}
\tikzset{
    circ/.style={circle, fill=black, inner sep=2pt, node contents={}}
}

\makeatletter
\renewcommand*\env@matrix[1][c]{\hskip -\arraycolsep
  \let\@ifnextchar\new@ifnextchar
  \array{*\c@MaxMatrixCols #1}}
\makeatother

\include{StartingPackage}
\usepackage[ruled,vlined]{algorithm2e}
\usepackage{protosem}

\newtheorem{theorem}{Theorem}
\numberwithin{theorem}{section}
\newtheorem{proposition}[theorem]{Proposition}
\newtheorem{lemma}[theorem]{Lemma}
\newtheorem{corollary}[theorem]{Corollary}

\theoremstyle{definition}

\newtheorem{remark}[theorem]{Remark}
\newtheorem{example}[theorem]{Example}
\newcommand\xqed[1]{%
  \leavevmode\unskip\penalty9999 \hbox{}\nobreak\hfill
  \quad\hbox{#1}}
\newcommand{\exampleend}{\xqed{$\triangle$}}

\numberwithin{equation}{section}

\crefname{equation}{}{}
\crefname{equation}{}{}
\crefname{figure}{Figure}{Figure}
\crefname{section}{Section}{Section}
\crefname{lemma}{Lemma}{Lemma}
\crefname{proposition}{Proposition}{Proposition}
\crefname{theorem}{Theorem}{Theorem}
\crefname{corollary}{Corollary}{Corollarie}
\crefname{definition}{Definition}{Definition}
\crefname{notation}{Notations}{Notation}
\crefname{remark}{Remark}{Remark}
\crefname{claim}{Claim}{Claim}
\crefname{assumption}{Assumption}{Assumption}

\newcommand{\realproj}{\mathbb R\mathrm P}

\DeclareMathOperator*{\mean}{\operatorname{\mathbb E}}

\newcommand{\R}{\mathbb{R}}
\renewcommand{\S}{\mathbb{S}}
\newcommand{\vol}{\operatorname{\mathrm{vol}}}
\newcommand{\bfu}{\mathbf{u}}
\newcommand{\bfv}{\mathbf{v}}
\newcommand{\bfx}{\mathbf{x}}
\newcommand{\bfy}{\mathbf{y}}
\newcommand{\bft}{\mathbf{t}}
\newcommand{\bfa}{\mathbf{a}}
\newcommand{\bfp}{\mathbf{p}}

\newcommand{\bfe}{\mathbf{e}}
\newcommand{\bfs}{\mathbf{s}}
\newcommand{\bfw}{\mathbf{w}}
\newcommand{\bfz}{\mathbf{z}}
\newcommand{\bfb}{\mathbf{b}}

\title{Average degree of the essential variety}

\author[Breiding]{Paul Breiding}
\address{Paul Breiding\\ Osnabr\"uck University, Germany}

\email{pbreiding@uni-osnabrueck.de}

\author[Fairchild]{Samantha Fairchild}
\address{Samantha Fairchild\\ MPI MiS, Germany}
\email{samantha.fairchild@mis.mpg.de}

\author[Santarsiero]{Pierpaola Santarsiero}
\address{
 Pierpaola Santarsiero\\
Osnabr\"uck University, Germany
}
\email{pierpaola.santarsiero@mis.mpg.de}

\author[Shehu]{Elima Shehu}
\address{Elima Shehu\\ MPI MiS and Osnabr\"uck University, Germany }
\email{elima.shehu@mis.mpg.de}

\date{}

\begin{document}

\maketitle

\begin{abstract}
The \emph{essential variety} is an algebraic subvariety of dimension $5$ in real projective space $\realproj^{8}$ which encodes the relative pose of two calibrated pinhole cameras. The $5$-point algorithm in computer vision computes the real points in the intersection of the essential variety with a linear space of codimension $5$. The degree of the essential variety is $10$, so this intersection consists of 10 complex points in general.

We compute the expected number of real intersection points when the linear space is random. We focus on two probability distributions for linear spaces. The first distribution is invariant under the action of the orthogonal group $\mathrm{O}(9)$ acting on linear spaces in $\realproj^{8}$. In this case, the expected number of real intersection points is equal to~$4$. The second distribution is motivated from computer vision and is defined by choosing 5 point correspondences in the image planes $\realproj^2\times \realproj^2$ uniformly at random. A Monte Carlo computation suggests that with high probability the expected value lies in the interval $(3.95 - 0.05,\ 3.95 + 0.05)$.

\vspace{.5cm}

\noindent\textbf{Keywords} 5 point relative pose problem, algebraic vision, random algebraic geometry, \hspace{.4cm}convex~geometry.

\end{abstract}

\section{Introduction}

The mathematical abstraction of a pinhole camera is a projective linear map
$$\realproj^3 \dashrightarrow \realproj^2, \quad \bfx\mapsto C\bfx,$$
where $C\in\mathbb R^{3\times 4}$ is a matrix of rank 3. The camera is called \emph{calibrated}, when $C=[R, \bft]$, where~$R\in\mathrm{SO}(3)$ is a rotation matrix and $\bft\in\mathbb R^3$ is a translation vector. 

The \emph{relative-pose problem} is the problem of computing the relative position of two cameras in 3-space; see \cite[Section 9]{hartley_zisserman_2004}.
Suppose that we have two calibrated cameras given by two matrices $C_1$ and~$C_2$ of rank 3. Since we are only interested in relative positions, we can assume $C_1=[\mathrm{1}_3, \mathbf{0}]$ and $C_2=[R, \bft]$. If $\bfx\in\realproj^3$ is a point in 3-space, $\bfu=C_1\bfx\in\realproj^2$ and~$\bfv=C_2\bfx\in\realproj^2$ are called a \emph{point-correspondence}. Any point-correspondence $(\bfu,\bfv)$ satisfies the algebraic equation
\begin{equation}\label{E_eq}
    \bfu^T E(R,\bft) \bfv = 0,\quad \text{ where } E(R,\bft) = [\bft]_\times \, R,
\end{equation}
and $[\bft]_\times$ is the matrix acting by $[\bft]_\times \bfx = \bft\times \bfx,$
the cross-product in $\mathbb R^3$. The set of all such matrices is denoted
$\widehat{\mathcal{E}} := \{E(R,\bft)\mid R \in \mathrm{SO}(3), \bft\in \R^3\}$. This is an algebraic variety defined by the 10 cubic and homogeneous polynomial equations $\det(E)=0,\; 2EE^TE - \mathrm{Tr}(EE^T)E=0$; see \cite[Section 4]{FM1990}.
Therefore, if $\pi: \R^{3\times 3} \mapsto \mathrm{P}(\mathbb{R}^{3\times 3})\cong \realproj^8$ denotes the projectivization map,~$\widehat{\mathcal{E}}$ is the cone over the projective variety
\begin{equation}\label{E_eq2}\mathcal{E} = \pi(\widehat{\mathcal{E}}),
\end{equation}
which is called the \emph{essential variety}. 

In the following we view elements in $\realproj^8$ as real $3\times 3$ matrices up to scaling.
The essential variety $\mathcal E$ is of dimension $5 = \dim \mathrm{SO}(3) + \dim \R^3 - 1$. Demazure showed that its complexification has degree $10$; see \cite[Theorem 6.4]{demazure:inria-00075672}. Denote by 
$\mathbb G:=G(3,\realproj^8)$
the Grassmannian of $3$-dimensional linear spaces in $\realproj^8$.
By \cref{E_eq}, every point correspondence induces a linear equation on $\mathcal E$. 
For 5 general point correspondences $(\bfu_1,\bfv_1),\ldots,(\bfu_5,\bfv_5)\in \realproj^2\times \realproj^2,$ the linear space $$L:=\{E\in \realproj^8 \mid \bfu_1^T E\bfv_1 = \cdots = \bfu_5^T E\bfv_5 = 0\}$$ is general in $\mathbb G$. Thus
$$\# (\mathcal E\cap L)\leq 10.$$
That is, the relative pose problem can be solved by computing the real zeros of a system of polynomial equations that has 10 complex zeros in general. 
Once we have computed $E=E(R,\bft)$ we can recover the relative position of the two cameras from $E$. The process of recovering the relative pose of two calibrated cameras from five point correspondences is known as the \emph{5-point algorithm}, {see \cite{N04}}.

The system of polynomial equations that we need to solve as part of the 5-point algorithm has 10 complex zeros in general, but the number of real zeros depends on $L$. Often, one computes all complex zeros and sorts out the real ones. Whether or not this is an efficient approach depends on how likely it is to have many real zeros out of 10 complex ones. Motivated by this observation, in this paper we study the \emph{average degree} $\mean \# (\mathcal E\cap L)$ for random $L$.

{Consider $L=U\cdot L_0$, where $L_0\in \mathbb G$ is fixed and ~$U\sim \mathrm{Unif}(\mathrm{O}(9))$ then with respect to Haar measure on $\mathbb G$ we in fact have $L\sim \mathrm{Unif}(\mathbb G)$; see \cite{ Kassel2019DeterminantalPM, PAUSINGER201913}.}
Our first result shows with this uniform distribution, we expect 4 of the 10 complex intersection points to be real.
\begin{theorem}\label{main1} 
{Let $L\sim \mathrm{Unif}(\mathbb G)$ then}$$\displaystyle\mean_{L\sim \mathrm{Unif}(\mathbb G)} \# (\mathcal E\cap L)=4.$$
\end{theorem}
This result is in fact quite surprising, because we get an integer, though there is no reason why it should even be a rational number (see also \cite[Remark 2]{BKL18}).

To work within the computer vision framework, we need a different distribution than used in \cref{main1}. The probability distribution is $\mathrm{O}(9)$-invariant, yet linear equations of the type $\bfu^TE\bfv=0$ are not $\mathrm{O}(9)$-invariant. These special linear equations are $\mathrm{O}(3)\times \mathrm{O}(3)$-invariant by the group action $(U,V).(\bfu,\bfv):=(U\bfu,V\bfv)$. The corresponding invariant probability distribution is given by the random point $\bfa=U\cdot \bfa_0\in\realproj^2$, where $U\sim \mathrm{Unif}(\mathrm{O}(3))$ and $\bfa_0\in\realproj^2$ is fixed. We denote this by $\bfa\sim\mathrm{Unif}(\realproj^2)$. 
\begin{remark}
The definition of $\mathrm{Unif}(\mathbb G)$ does not depend on the choice of $L_0$, and the definition of $\mathrm{Unif}(\realproj^2)$ does not depend on the choice of $\bfa_0$.
\end{remark}

We write $L\sim \psi$, where $L=\{E\in\realproj^8 \mid \bfu_1^T E\bfv_1 = \cdots = \bfu_5^T E\bfv_5 = 0\}\in \mathbb G$ is the random linear space given by i.i.d.\ points $\bfu_1,\bfv_1,\ldots,\bfu_5,\bfv_5\sim \mathrm{Unif}(\realproj^2)$. 
We have the following result.
\begin{theorem}\label{main2} With the distribution $\psi$ defined above,
$$\mean_{L\sim \psi} \# (\mathcal E\cap L)=\frac{\pi^3}{4} \cdot \mean  \left\vert\det \begin{bmatrix} \bfz_1 & \bfz_2 &\bfz_3&\bfz_4 & \bfz_5\end{bmatrix}\right\vert,$$
where 
$\bfz_1,\bfz_2,\bfz_3,\bfz_4, \bfz_5\sim \bfz$ are i.i.d., 
$$\bfz=
\begin{bmatrix}
b\cdot r\cdot \sin\theta, &
b\cdot r\cdot \cos \theta,&
a \cdot s \cdot \sin\theta,&
a \cdot s \cdot \cos\theta,&
rs
\end{bmatrix}^T\in\mathbb R^5$$
and $a,b,r,s\sim N(0,1)$, $\theta\sim \mathrm{Unif}([0,2\pi))$ are independent.
\end{theorem}
We were not able to determine the exact value of the integral in this theorem. Yet, we can independently sample $N$ random matrices of the form $\begin{bmatrix} \bfz_1 & \bfz_2 &\bfz_3&\bfz_4 & \bfz_5\end{bmatrix}$ and compute their absolute determinants.  This gives an empirical average value $\mu_N$. An experiment with sample size $N=5\cdot 10^9$ gives an empirical average of 
$$\mu_N \approx  3.95$$
In fact, $\mu_N$ is itself a random variable and we have $P\left(\ \vert \mu_N - \mean_{L\sim \psi} \# (\mathcal E\cap L)\vert \geq \varepsilon\ \right) \leq \frac{\pi^6}{16}\cdot \frac{\sigma^2}{N\cdot \varepsilon^2}$ by Chebychev's inequality, where $\sigma^2$ is the variance of the absolute determinant. We show in \cref{prop_var} below that $\sigma^2\leq 360$. Using this in Chebychev's inequality we get
$$P\left(\ \vert \mu_N - \mean_{L\sim \psi} \# (\mathcal E\cap L)\vert \geq 0.05\ \right) \leq  0.0175\%$$ (in fact, since $360$ is an extremely coarse upper bound, the true probability should be much smaller). Therefore, it is likely that~$\mean_{L\sim \psi} \# (\mathcal E\cap L)$ is strictly smaller than 4; i.e., it is likely that the expected value in \cref{main2} is less than the one in \cref{main1}. See \cref{fig:exp}.

The distribution of zeros shown in \cref{fig:exp} gives rise to further questions of interest in computer vision. When applying the 5-point algorithm it is important to know when there are no real solutions. In \cref{fig:exp}, for 1000 sampled spaces, the distribution with respect to $\mathrm{Unif}(\mathbb{G})$ had 10 instances with no real solutions, and the distribution with $\psi$ had only 1 instance with no real solutions. The experiments give an indication that no real solutions is a relatively rare occurrence, but further work will need to be done to quantify and geometrically characterize these occurrences with respect to different distributions.

We remark that the distributions $\mathrm{Unif}(\mathbb G)$ and $\psi$ are different in the following sense. For $L\sim \mathrm{Unif}(\mathbb G)$ every linear space~$L\in \mathbb G$ has the same probability. But when $L\sim \psi$, it must be defined by $5$ linear equations that are given by rank-one matrices of size $3$. The Segre variety of rank-one matrices of size $3$ in~$\realproj^8$ has dimension~4 (see\footnote{In \cite{Landsberg2012} one can find a formula for the dimension of the complex Segre variety. The real Segre variety is Zariski dense in the complex Segre variety, so their real and complex dimensions coincide.} \cite[Section 4.3.5]{Landsberg2012}), so that a general linear space of codimension $4=9-5$ in~$\realproj^8$, spanned by 5 general $3\times 3$ matrices, intersects the Segre variety in finitely many points. There is an Euclidean open subset in~$\mathbb G$, such that this intersection has strictly less than 5 points. Hence, there is a measurable subset~$\mathcal W\subset \mathbb G$ such that $P_{L\sim \mathrm{Unif}(\mathbb G)}(L\in \mathcal W)>0$ but~$P_{L\sim \psi}(L\in \mathcal W)=0$.

In \cref{sec_zonoid} we use a result by Vitale \cite{Vitale} to express the expected value in \cref{main2} through the volume of a certain convex body $K\subset \R^5$. Namely, 
\begin{equation}\label{main2_K}
\mean_{L\sim \psi} \# (\mathcal E\cap L) = 30\pi^2 \cdot \mathrm{vol}(K),
\end{equation}
and $K$ defined by its support function 
$h_K(\bfx)= \tfrac{1}{2}\mean_{\bfz} \vert \bfx^T\bfz\vert$, and $\bfz\in\R^5$ is as above; $K$ is a zonoid and we call it the \emph{essential zonoid}. We use this to prove a lower bound for the expected number of real points~$\mean_{L\sim \psi} \# (\mathcal E\cap L)$ in \cref{thm: lower bound}.

The two probability distributions in \cref{main1} and \cref{main2} are \emph{geometric}, meaning that they are not biased towards preferred points in $\mathbb G$ or $\realproj^2$, respectively.  In applications, however, one might be interested in other distributions, like for instance taking the $\bfu_i$ and~$\bfv_i$ uniformly in a box (see Examples \ref{ex1} and \ref{ex11} below). For such a case, we do not get concrete results like \cref{main1} or \cref{main2}. Nevertheless, in \cref{main3} below we give a general integral formula.

\begin{figure}[ht]
\includegraphics[width = 0.49\textwidth]{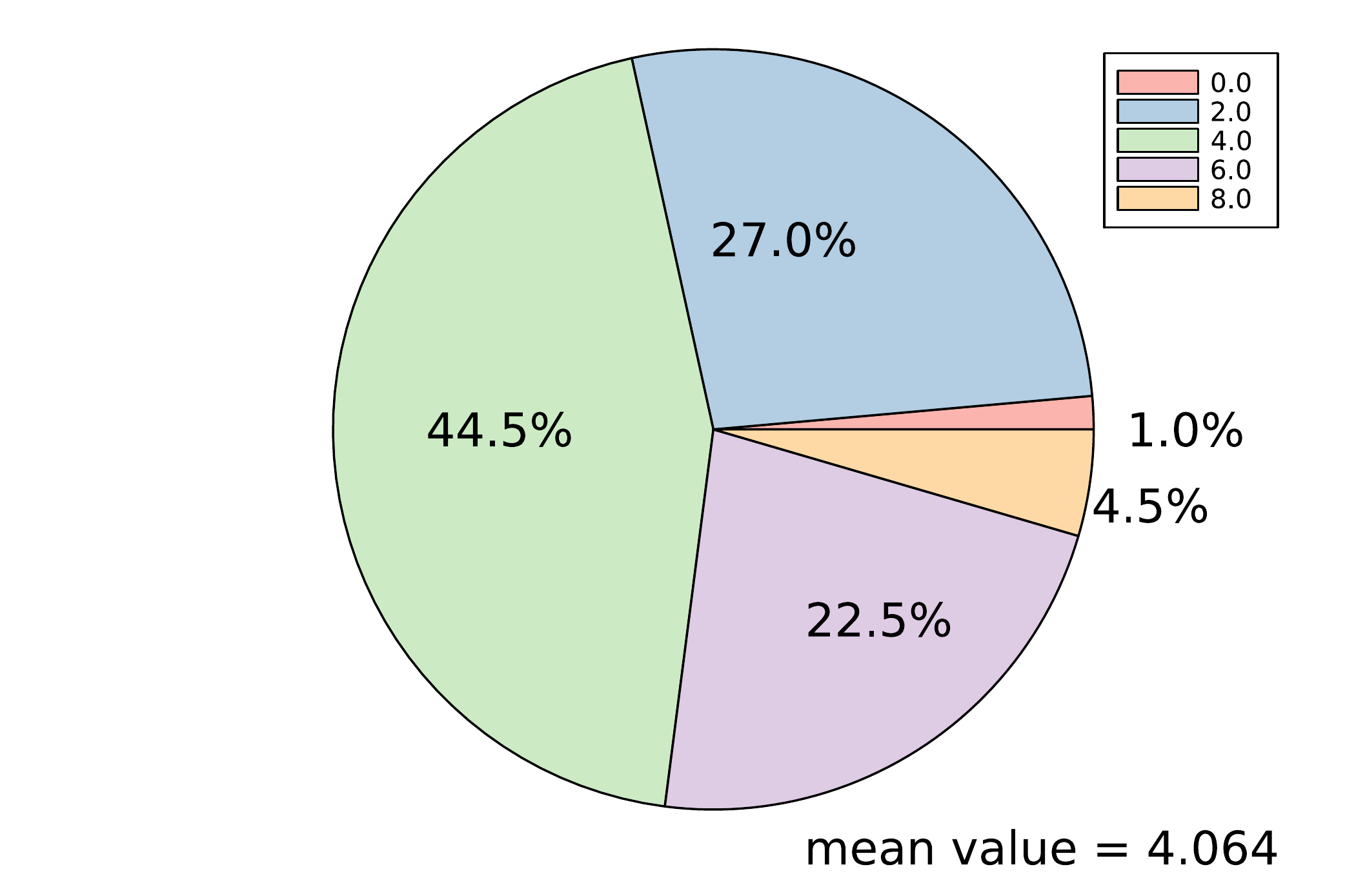}
\hfill
\includegraphics[width = 0.49\textwidth]{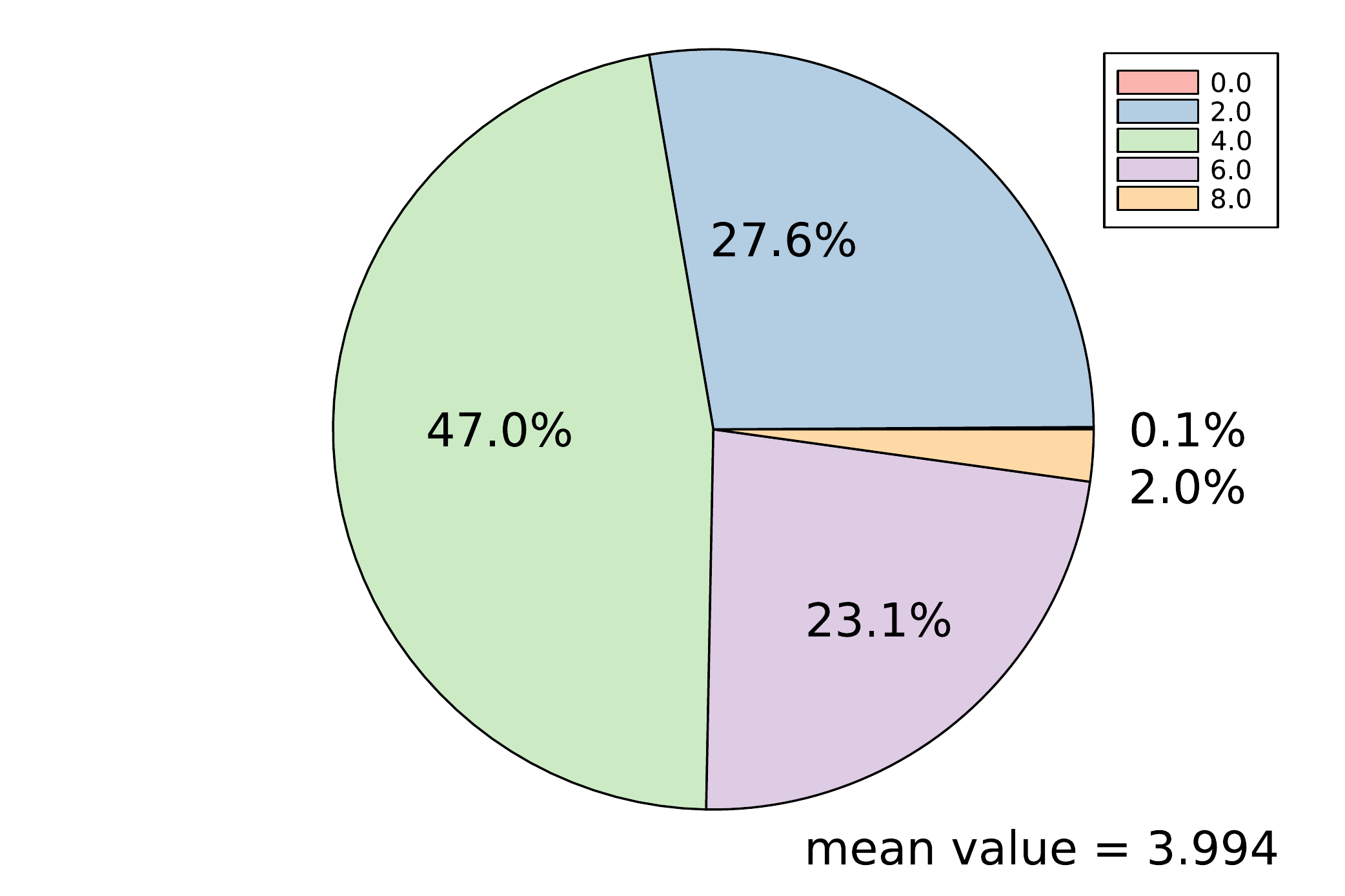}
\caption{The two pie charts show the outcome of the following two experiments. We sampled $N=1000$ random linear spaces, once with distribution $\mathrm{Unif}(\mathbb G)$ (the left chart) and once with distribution $\psi$ (the right chart). Then, we computed $\mathcal E\cap L$ by solving the system of polynomial equations with the software \texttt{HomotopyContinuation.jl} \cite{HC.jl}. The charts show the empirical distribution of real zeros and the corresponding empirical means in these experiments.}
\label{fig:exp}
\end{figure}

\subsection*{Outline} 
In \cref{sec:preliminaries} we give preliminaries. We recall the integral geometry formula in projective space and study the geometry of the essential variety. In \cref{sec:volume} we prove \cref{main1} by computing the volume of the essential variety. In \cref{sec:proof_main2} we prove \cref{main2} and \cref{main3}. In the last section, \cref{sec_zonoid}, we study the essential zonoid.

\medskip

\section{Preliminaries}\label{sec:preliminaries}
Let us start by setting up our notation as well as making note of many key volume computations used throughout the paper. We consider the Euclidean space $\mathbb R^n$ with the standard metric $\langle\bfx, \bfy\rangle = \bfx^T\bfy$. The norm of a vector $\bfx\in\mathbb R^3$ will be denoted by $\Vert \bfx\Vert:=\sqrt{\langle \bfx,\bfx\rangle}$ and the unit sphere by $\S^{n-1}:=\{\bfx\in\mathbb R^n \mid \Vert \bfx \Vert = 1\}$.  
The Euclidean volume of the sphere is 
\begin{equation}\label{volume_sphere}
\vol(\S^n) = \frac{2 \pi^{\frac{n+1}{2}}}{\Gamma\left(\frac{n+1}{2}\right)}.\end{equation}
In particular $\vol(\S^1) =2\pi$ and  $\vol(\S^2) = 4\pi$. The standard basis vectors in $\mathbb{R}^n$ are denoted~$\bfe_i$ for~$1\leq i\leq n$.
The space of real $n\times n$ matrices $\mathbb R^{n\times n}$ is also endowed with a Euclidean structure
$$\langle A,B\rangle := \frac{1}{2}\, \mathrm{Tr}(AB^T),\quad A,B\in\mathbb R^{n\times n}.$$
We denote the identity matrix $\mathrm{1}_n\in\mathbb R^{n\times n}$ {and the zero matrix $0_n$.} The orthogonal group will be denoted by $\mathrm{O}(n)$, while the special orthogonal group is
$\mathrm{SO}(n)$.
Both the orthogonal and special orthogonal group are Riemannian submanifolds of $\mathbb R^{n\times n}$.
Volumes of the two manifolds are  $$\vol(\mathrm{O}(n)) = 2\prod_{k=1}^{n-1} \mathrm{vol}(\S^k) \quad\hbox{ and } \quad \vol(\mathrm{SO}(n)) = \frac{1}{2}\vol(\mathrm{O}(n));$$
see \cite[Equation (3-15)]{Howard}. For instance, $\mathrm{vol}(\mathrm{SO}(2)) = 2\pi$ and $\mathrm{vol}(\mathrm{SO}(3)) = 8\pi^2$.

\subsection{Integral geometry}

The \emph{real projective space}  of dimension $n-1$ is defined to be
$\realproj^{n-1} := (\mathbb R^{n}\setminus \{0\})/\sim$, where the equivalence relation is $\bfx\sim \bfy \Leftrightarrow \exists \lambda \in\mathbb R: \bfx=\lambda \bfy$. The projection $\pi:\S^{n-1} \to \realproj^{n-1}$ that maps $\bfx$ to its class is a $2:1$ cover. It induces a Riemannian structure on $\realproj^{n-1}$ by declaring $\pi$ to be a local isometry.

Let now $X\subseteq \realproj^{n-1}$ be a submanifold of dimension $m$ and $L\subseteq \realproj^{n-1}$ be a linear space of codimension $m$. Howard \cite{Howard} proved that for almost all $U\in \mathrm{O}(n)$ we have that $X\cap U \cdot L$ is finite and
\begin{equation}\label{IG_formula}
\mean_{U\sim \mathrm{Unif}(\mathrm{O}(n))} \vol(X\cap U\cdot L)  = \frac{\vol(X)}{\vol(\realproj^{m})};
\end{equation}
see \cite[Theorem 3.8 \& Corollary 3.9]{Howard}.
This formula will be used for proving \cref{main1}.

\subsection{The coarea formula}
The proof of \cref{IG_formula} is based on the coarea formula, which we will also need. In order to state the formula we need to introduce the normal Jacobian. Let $M, N$ be Riemannian manifolds with $\dim(M)\geq \dim(N)$ and let $F\colon M\rightarrow N$ be a surjective smooth map. Fix a point $\bfx\in M$. The \emph{normal Jacobian} $\mathrm{NJ}(F,\bfx)$ of $F$ at $\bfx$ is 
$$\mathrm{NJ}(F,\bfx)= \sqrt{\det JJ^T},$$
where $J$ is the matrix representation of the derivative $\mathrm D_\bfx F$ relative to orthonormal bases in $T_\bfx M$ and $T_{F(\bfx)}N$. Then for any integrable function $h:M\to \R$ 
\begin{equation}\label{coarea_formula}\int_M h(\bfx) \,\mathrm d\bfx = \int_{\bfy\in N} \left(\int_{\bfx\in F^{-1}(\bfy)} \frac{h(\bfx)}{\mathrm{NJ}(F,\bfx)} \,\mathrm d\bfx\right)\,\mathrm d\bfy.
\end{equation}
See, e.g., \cite[Section A-2]{Howard}.

\subsection{The geometry of the essential variety}

In this subsection, we study in more detail the geometry of the essential variety $\mathcal E$. Recall from \cref{E_eq2} that $\mathcal E$ is the projection of the cone~$\widehat{\mathcal E}$ to projective space $\realproj^8$. We can also project $\widehat{\mathcal E}$ to the sphere. This defines the \emph{spherical essential variety}
$$\mathcal E_{\mathbb S} := \{E\in \widehat{\mathcal E} \mid \Vert E\Vert = 1\}.$$
Recall from \cref{E_eq} the definition of $E(R, \bft)$.

\begin{lemma}\label{lemma:image_essential}
The map $E:  \mathrm{SO}(3)\times  \S^2 \to \R^{3\times 3}, (R,\bft)\mapsto E(R,\bft)$ is 2:1 and $\operatorname{im}(E) =\mathcal E_{\mathbb S}$.
\end{lemma}
\begin{proof}
Let $(R,\bft)\in \mathrm{SO}(3)\times  \S^2 $. The matrix description of $[\bft]_\times$ is 
$$[\bft]_\times = \begin{bmatrix} 0 & -t_3 & t_2 \\ t_3 & 0 & -t_1 \\ -t_2 & t_1 & 0\end{bmatrix}.$$
In particular, this shows $\mathrm{Tr}\left([\bft]_\times [\bft]_\times^{T}\right) = 2\Vert \bft\Vert^2 = 2$.
Then, the norm squared of $E(R,\bft)$ is 
$$\Vert E(R,\bft)\Vert^2 
= \frac{1}{2}\,\mathrm{Tr}\left(\ E(R,\bft)E(R,\bft)^T\ \right) 
= \frac{1}{2}\,\mathrm{Tr}\left(\ [\bft]_\times R R^T [\bft]_\times^T\ \right) 
= \frac{1}{2}\,\mathrm{Tr}\left(\ [\bft]_\times [\bft]_\times^{T}\ \right) 
= 1.
$$
Therefore, $\operatorname{im}(E) =\mathcal E_{\mathbb S}$. 
Let $M\in\mathrm{SO}(3)$ be a matrix such that $M\bft = \bft$ and $M\bfx =-\bfx$ for all~$\bfx$ orthogonal to $\bft$, then we~have $M[-\bft]_\times = [\bft]_\times$ and we can write the following 
\begin{equation}\label{eq20}E(R,\bft) = [\bft]_\times R =[\bft]_\times M^TM R=(M[-\bft]_\times)^T MR= [-\bft]_\times MR =E(MR, -\bft). 
\end{equation}
This means that $E$ is {at least} 2:1. {To show it is at most $2:1$, 
we consider the following
\[
[\bft]_\times R = [\lambda \bft]_\times MR,
\]
for some rotation $M$ and $\lambda \in \{\pm 1\}$. We want to check how many different rotation matrices  $M$ satisfy this equation. We have the following chain of implications \[
[\bft]_\times R = [\lambda \bft]_\times MR\,\, \Longrightarrow \,\,[\bft]_\times = [\lambda \bft]_\times M
\,\,\Longrightarrow\,\, [\bft]_\times (1_3 - \lambda M) = 0.
\] We see that the columns of $1_3 - \lambda M$ are multiples of $\bft,$ therefore we can write $1_3 - \lambda M = c\,\bft\bft^T$ for some $c \in \mathbb{R}$. We make use of the fact that $\det(M)=1.$ Firstly we compute the determinant
\[
\det(M)=\det(\lambda^{-1}(1_3-c\bft\bft^T)) = \lambda^{-3} \det(1_3-c\,\bft\bft^T) = \lambda^{-3}(1-c) ,
\]
where we have used that $\bft^T\bft=1$. This implies
$\lambda^{3} = 1 - c$. If $\lambda = 1$, then $c=0$. If $\lambda = -1$, then we have $c = 2$. Thus, either $M = 1_3$ or $M = 2\bft\bft^T - 1_3$.

This is Rodrigues’ formula for 180-degree rotation about the axis spanned by $\bft$. Additionally, it is worth mentioning that this symmetry of the essential variety is exactly the \textit{twisted pair}, described in \cite{hartley_zisserman_2004}. } \end{proof}

Next, we show the invariance properties of the map $E$. For $U,V\in\mathrm{SO}(3)$ we denote
$$(U,V).E := U\, E\, V^T.$$
In particular, the next lemma shows that this defines a group action on $\mathcal E_{\S}$.
\begin{lemma}\label{lemma:invariance_phi}
For orthogonal matrices $U,V\in\mathrm{SO}(3)$ and $(R,\bft)\in \mathrm{SO}(3)\times \S^2$  we have
$$E(URV^T, U\bft) = (U,V).E(R,\bft).$$
\end{lemma}
\begin{proof}
We have $E(URV^T, U\bft)  = [U\bft]_\times URV^T =  ([U\bft]_\times UR)V^T$. Moreover, the cross product satisfies $(U\bft)\times (U\bfx)=U(\bft \times \bfx)$ for all $\bfx\in\mathbb R^3$. 
\end{proof}
With the above lemma, we deduce the following result on $\mathcal{E}_{\S}$.
\begin{corollary}\label{cor_hom_space}
$\mathcal E_{\mathbb S}$ is a homogeneous space for $\mathrm{SO}(3)\times\mathrm{SO}(3)$ acting by left and right multiplication. In particular,
$\mathcal E_{\mathbb S}$, and hence also $\mathcal E$, is smooth.
\end{corollary} 
We now denote the following special matrix in $\mathcal E$:
\begin{equation}\label{def_E_0}
E_0 := E(\mathrm{1}_3, \bfe_1) 
= \begin{bmatrix}
    0 & 0 & 0\\
    0 & 0 & -1\\
    0 & 1 & 0\\
\end{bmatrix}
\end{equation}
(recall that $\bfe_1$ denotes the first standard basis vector $(1,0,0)^T$).

\begin{lemma}\label{lem_stabilizer}
The stabilizer group of $E\in\mathcal E_{\S}$ under the $\mathrm{SO}(3)\times\mathrm{SO}(3)$ action has volume equal to~$2\sqrt{2} \cdot \mathrm{vol}(\mathrm{SO}(2))$.
\end{lemma}
\begin{proof}
The stabilizer groups of $E\in\mathcal E_{\S}$ all have the same volume. We compute the stabilizer group of $E_0$. By \cref{lemma:image_essential},  $E(R,\bft)$ is 2:1 and by \cref{eq20} we have
$$E_0 = E(\mathbf 1_3, \bfe_1) = E(M, -\bfe_1),$$
where $M=\left[\begin{smallmatrix} 1 & 0 & 0 \\ 0 & -1 & 0 \\  0 & 0&-1\end{smallmatrix}\right]$.
Therefore, $(U,V).E_0 = E_0$ if and only if~$U\bfe_1 = \bfe_1$ and $UV^T = 1_3$, or $U\bfe_1 = -\bfe_1$ and $UV^T = M$; i.e., $MU=V$. That is, $\mathrm{stab}(E_0)$ is realized as the image of the map
${F: \mathrm{SO}(2)\times\{-1,1\}\rightarrow\mathrm{SO}(3)\times\mathrm{SO}(3)}$ {such that}
$$(\tilde{U},\varepsilon)\mapsto\left(\begin{bmatrix}
    \varepsilon & 0 & 0\\ 0 &  u_{11} & u_{12}\\ 0 &  u_{21} & u_{22}
\end{bmatrix},\begin{bmatrix}
    \varepsilon & 0 & 0\\ 0 & \varepsilon u_{11} & \varepsilon u_{12}\\ 0 & \varepsilon u_{21} &\varepsilon u_{22}
\end{bmatrix}\right), \mbox{ where } \tilde{U}=\begin{bmatrix} u_{11} & u_{12}\\  u_{21} & u_{22} \end{bmatrix} \in \mathrm{SO}(2),\,\varepsilon\in\{-1,1\}.$$  
{ The normal Jacobian of $F$ at every point is  $\sqrt{2}.$ For fixed $\varepsilon$, $\mathrm{SO}(2)\times \{\varepsilon\}$ is a homogeneous space under the action of $\mathrm{SO}(2)$ acting on itself. This group action is transitive and preserves the inner product, so the normal Jacobian is constant. Thus it suffices to compute the normal Jacobian at $(1_2,\varepsilon)$. To see this, the tangent space to $\mathrm{SO}(3)$ at the identity is $$T_{\mathrm{1}_3}\mathrm{SO}(3)=\mathrm{span}\{F_{1,2}, F_{1,3}, F_{2,3}\}$$ for $F_{i,j}=\bfe_i\bfe_j^T - \bfe_j\bfe_i^T$. Thus an orthogonal basis for the tangent space of~$\mathrm{SO}(3)\times \mathrm{SO}(3)$ at $(\mathrm 1_3,\mathrm 1_3)$, is given by \begin{equation} \label{eq:SO3basis}
\{(0_3, F_{1,2}), (0_3, F_{1,3}), (0_3, F_{2,3}), (F_{1,2},0_3), (F_{1,3},0_3), (F_{2,3},0_3)\}.
\end{equation}
Indeed, with respect to this basis and identifying the tangent space of $\mathrm{SO}(2) \times \{-1,1\}$ with $\R$, we have $D_{(\mathrm{1}_2,
\varepsilon)}F = \begin{bmatrix} 0&0&\varepsilon&0&0&1\end{bmatrix}^T$ and thus 
$$NJ(F, (\mathrm{1}_2,1)) = \left\|\begin{bmatrix} 0&0&\varepsilon&0&0&1\end{bmatrix}^T\right\|= \sqrt{2}.$$ 

We conclude by using the coarea formula \cref{coarea_formula} for $M= \mathrm{SO}(2) \times \{-1,1\},$ $N= \mathrm{stab}(E_0)$, $h \equiv  \sqrt{2}$, and $F^{-1}(y)$ a single point by injectivity to obtain $\mathrm{vol}(\mathrm{stab}(E_0)) = 2\sqrt{2} \cdot \mathrm{vol}(\mathrm{SO}(2)).$
}
\end{proof}

Next, we compute an orthonormal basis of the tangent space $T_{E_0}\mathcal{E}$ at $E_0$.
\begin{lemma}\label{prop_TS}
An orthonormal basis of $T_{E_0}\mathcal E$ is given by the following five matrices
\begin{alignat*}{3}
&B_1 = \begin{bmatrix}
    0 & 0 & 0\\
    0 & 0 & 0\\
    \sqrt{2} & 0 & 0\\
\end{bmatrix},\quad
&&B_2 = \begin{bmatrix}
    0 & 0 & 0\\
    \sqrt{2} & 0 & 0\\
    0 & 0 & 0\\
\end{bmatrix},\quad && B_3 =\begin{bmatrix}
    0 & 0 & \sqrt{2}\\
    0 & 0 & 0\\
    0 & 0 & 0\\
\end{bmatrix},
\\[0.7em]
 &B_4 = \begin{bmatrix}
    0 & \sqrt{2} & 0\\
    0 & 0 & 0\\
    0 & 0 & 0\\
\end{bmatrix},\quad
&&B_5 = \begin{bmatrix}
    0 & 0 & 0\\
    0 & 1 & 0\\
    0 & 0 & 1\\
\end{bmatrix}{.}
&&
\end{alignat*}
\end{lemma}
\begin{proof} 
First, we observe that the five matrices above are pairwise orthogonal and all of norm one. Since $\dim \mathcal E=5$, it therefore suffices to show that $B_1,\ldots,B_5\in T_{E_0}\mathcal E =  T_{E_0}\mathcal E_{\S}$. 
The derivatives of $E$ evaluated in $(\mathrm{1}_3, \dot \bft)$ and $(\dot R, \bfe_1)$ respectively are
\begin{align*}
    \frac{\partial E}{\partial \bft}(\mathrm{1}_3,\dot{\bft}) &=E(\mathrm{1}_3, \dot{\bft}),\quad
\frac{\partial E}{\partial R}(\dot R, \bfe_1)= E(\dot R, \bfe_1).
\end{align*}
We have $T_{\bfe_1} \S^2 = \mathrm{span}\{\bfe_2, \bfe_3\}$ and $T_{\mathrm{1}_3}\mathrm{SO}(3)=\mathrm{span}\{F_{1,2}, F_{1,3}, F_{2,3}\}$, where $F_{i,j}=\bfe_i\bfe_j^T - \bfe_j\bfe_i^T$ as above. Therefore, the following five matrices are in $T_{E_0}\mathcal E$:
\begin{alignat*}{2}
E(\mathrm{1}_3, \bfe_2)=&\begin{bmatrix}
    0 & 0 & 1\\
    0 & 0 & 0\\
    -1 & 0 & 0\\
\end{bmatrix}, \quad E(\mathrm{1}_3,\bfe_3) =&&\begin{bmatrix}
    0 & -1 & 0\\
    1 & 0 & 0\\
    0 & 0 & 0\\
\end{bmatrix}{,}\\[0.5em]\nonumber
E(F_{1,2}, \bfe_1) =&
\begin{bmatrix}
    0 & 0 & 0\\
    0 & 0 & 0\\
    -1 & 0 & 0\\
\end{bmatrix}, \quad E(F_{1,3}, \bfe_1) =&&
\begin{bmatrix}
    0 & 0 & 0\\
    1 & 0 & 0\\
    0 & 0 & 0\\
\end{bmatrix}, \quad E(F_{2,3}, \bfe_1) = \begin{bmatrix}
    0 & 0 & 0\\
    0 & 1 & 0\\
    0 & 0 & 1\\
\end{bmatrix}.
\end{alignat*}
Each of the $B_i$ above can be expressed as a linear combination of these five matrices, which shows $B_i\in T_{E_0}\mathcal E$.
\end{proof}
Alternatively, to prove \cref{prop_TS} we consider the derivative of the smooth surjective map $\gamma: \mathrm{SO}(3)\times \mathrm{SO}(3) \to \mathcal{E}_{\S}, (U, V)\mapsto (U, V).E_0$. Since the basis for the tangent space of $\mathrm{SO}(3)\times \mathrm{SO}(3)$ at $(\mathrm 1_3,\mathrm 1_3)$ is given as in \cref{eq:SO3basis}, the tangent space $T_{E_0}\mathcal E$ is also spanned by the following six matrices

\begin{alignat}{2} \label{matrices2}
    E_0F_{1,2}^T=&\begin{bmatrix}
        0 & 0 & 0\\
        0 & 0 & 0\\
        1 & 0 & 0\\
    \end{bmatrix}, \quad E_0F_{1,3}^T =&& \begin{bmatrix}
        0 & 0 & 0\\
        -1 & 0 & 0\\
        0 & 0 & 0\\
    \end{bmatrix},\quad E_0F_{2,3}^T =
    \begin{bmatrix}
        0 & 0 & 0\\
        0 & -1 & 0\\
        0 & 0 & -1\\
    \end{bmatrix},\\[0.5em]\nonumber
    F_{1,2}E_0 =&
    \begin{bmatrix}
        0 & 0 & -1\\
        0 & 0 & 0\\
        0 & 0 & 0\\
    \end{bmatrix}, \quad   F_{1,3}E_0 =&&
    \begin{bmatrix}
        0 & 1 & 0\\
        0 & 0 & 0\\
        0 & 0 & 0\\
    \end{bmatrix}, \quad    F_{2,3}E_0 = \begin{bmatrix}
        0 & 0 & 0\\
        0 & 1 & 0\\
        0 & 0 & 1\\
    \end{bmatrix}.
    \end{alignat}
\medskip

\section{The volume of the essential variety}\label{sec:volume}
In this section, we prove \cref{main1}. The strategy is as follows. By \cref{cor_hom_space}, $\mathcal E$ is a smooth submanifold of $\realproj^8$. We can apply the integral geometry formula \cref{IG_formula} to get 
\begin{equation}\label{IG_in_our_case}
\mean_{L\sim \mathrm{Unif}(\mathbb G)} \vol(\mathcal E\cap L)  = \frac{\vol(\mathcal E)}{\vol(\realproj^{5})}.
\end{equation}
Thus, to prove \cref{main1} we can compute the volume of $\mathcal E$. We do this in the next theorem. Notice that the result of the theorem, when plugged into \cref{IG_in_our_case} immediately, proves \cref{main1}.

\begin{theorem}\label{lem:volEssential} The volume of the essential variety is
    $$\vol(\mathcal{E})= 4 \cdot  \mathrm{vol}(\realproj^5). $$
\end{theorem}
We give two different proofs of this theorem.
Since $\vol(\mathcal{E})=\tfrac{1}{2}\,\vol(\mathcal{E}_{\S})$, it is enough to compute the latter volume.

\begin{proof}[Proof 1]
By Lemma~\ref{lemma:image_essential}, we realize $\mathcal{E}_\S$ as the image of the smooth map
$(R,\bft)\mapsto E(R,\bft)$, and we now restrict the domain to the image. By Lemma \ref{lemma:invariance_phi}, $\mathrm{NJ} (E, (R,\bft))$ is invariant under the action by $\mathrm{SO}(3)\times\mathrm{SO}(3)$. 
Applying the coarea formula \cref{coarea_formula} over the 2-element fibers of $E$, we get that
$$\vol(\mathcal{E}_{\S}) = \int_{\mathcal E_{\mathbb S}} 1\; \mathrm d E = \frac{1}{2}
\int_{\mathrm{SO}(3)\times \S^2}\mathrm{NJ} (E, (R,\bft)) \; \mathrm d(R,\bft).$$ This implies
 \begin{align*} \label{eq:allbutJ}
   \vol(\mathcal{E}_{\S}) = & \frac{1}{2} \vol(\mathrm{SO}(3)) \cdot \vol(\S^2) \cdot \mathrm{NJ} (E, (\mathrm 1_3,\bfe_1))= 16\pi^3 \cdot \mathrm{NJ} (E, (\mathrm 1_3,\bfe_1)).
   \end{align*}
Recall, $F_{i,j}=\bfe_i\bfe_j^T - \bfe_j\bfe_i^T$.
With respect to the orthonormal basis $\{B_i\}$ computed in Lemma~\ref{prop_TS} and the orthonormal basis $\{(0_3,\bfe_2), (0_3,{\bfe_3}), ( F_{1,2}, {\bf 0}), (F_{1,3},{\bf0} ), ( F_{2,3},{\bf0})\}$ computed for~$ T_{\mathrm 1_3}\mathrm{SO}(3)\times T_{\bfe_1}\S^2$,  the columns of the matrix $J$ associated to the derivative of $E$ at $(\mathrm 1_3,\bfe_1)$ are the basis elements of $ T_{\mathrm 1_3}\mathrm{SO}(3)\times T_{\bfe_1}\S^2$ written as a combination of the basis given by Lemma~\ref{prop_TS}:
\begin{align*}
J 
&=  \frac{1}{\sqrt{2}}\begin{bmatrix}
-1 & 0 & -1 & 0 & 0\\
0 & 1 & 0 & 1 & 0\\
1 & 0 & 0 & 0 & 0\\
0 & -1 & 0 & 0 & 0\\
0 & 0 & 0 & 0 & \sqrt{2}
\end{bmatrix}.
\end{align*}
So, we have that $\mathrm{NJ} (E, (\mathrm 1_3,\bfe_1)) = \sqrt{\det JJ^T} = \frac{1}{4}$, and consequently $\vol(\mathcal{E}_{\S}) = 4\pi^3$. Therefore, we have $\vol(\mathcal{E}) = 2\pi^3$. By \cref{volume_sphere}, $\mathrm{vol}(\realproj^5) = \frac{1}{2}\cdot \mathrm{vol}(\S^5) = \frac{\pi^3}{2}$, so $\mathrm{vol}(\mathcal E) = 4 \cdot \mathrm{vol}(\realproj^5)$.
\end{proof}

\begin{proof}[Proof 2]
By \cref{cor_hom_space}, $\mathcal{E}_{\S}$ is a homogeneous space under the action of $\mathrm{SO}(3)\times \mathrm{SO}(3)$. We therefore have the surjective smooth map $\gamma: \mathrm{SO}(3)\times \mathrm{SO}(3) \to \mathcal{E}_{\S}, (U,V)\mapsto (U,V).E_0$ with fibers that satisfy $\mathrm{vol}(\gamma^{-1}(E)) = 2\sqrt{2}\cdot \mathrm{vol}(\mathrm{SO}(2))$ for all $E\in\mathcal E_{\S}$; see \cref{lem_stabilizer}. The coarea formula from \cref{coarea_formula} implies
$$\mathrm{vol}(\mathcal{E}_{\S})\cdot 2\sqrt{2}\cdot \mathrm{vol}(\mathrm{SO}(2)) = \int_{\mathrm{SO}(3)\times \mathrm{SO}(3)} \mathrm{NJ}(\gamma, (U,V))\; \mathrm d(U,V).$$
By \cref{lemma:invariance_phi}, the map $\gamma$ is equivariant with respect to the $\mathrm{SO}(3)\times \mathrm{SO}(3)$ action. This implies, that the value of the normal Jacobian does not depend on $(U,V)$. Therefore, we have 
$\mathrm{vol}(\mathcal{E}_{\S})\cdot2\sqrt{2}\cdot \mathrm{vol}(\mathrm{SO}(2)) =  \mathrm{NJ}(\gamma, (\mathrm 1_3,\mathrm 1_3)) \cdot \mathrm{vol}(\mathrm{SO}(3))^2,$ and
so 
$$\mathrm{vol}(\mathcal{E}_{\S}) =  \frac{\mathrm{vol}(\mathrm{SO}(3))^2}{2\sqrt{2}\cdot\mathrm{vol}(\mathrm{SO}(2))} \cdot \mathrm{NJ}(\gamma, (\mathrm 1_3,\mathrm 1_3)) = \frac{16\pi^3}{\sqrt{2}}\cdot \mathrm{NJ}(\gamma, (\mathrm 1_3,\mathrm 1_3)).$$
We compute the normal Jacobian. Recall the notation $F_{i,j}=\bfe_i\bfe_j^T - \bfe_j\bfe_i^T$. 

With respect to the orthonormal basis computed in Lemma~\ref{prop_TS} and the orthonormal basis as in \cref{eq:SO3basis} for the tangent space of $\mathrm{SO}(3)\times \mathrm{SO}(3)$ at $(\mathrm 1_3,\mathrm 1_3)$, the columns of the matrix $J$ associated to the derivative of $\gamma$ at $(\mathrm 1_3,\mathrm 1_3)$ are given by writing the matrices in \cref{matrices2} with respect to the basis in Lemma~\ref{prop_TS}:
\begin{align*}J 
&= \frac{1}{\sqrt{2}}\begin{bmatrix}
1&0&0&0&0&0\\
0&-1&0&0&0&0\\
0&0&0&-1&0&0\\
0&0&0&0&1&0\\
0&0&-\sqrt{2}&0&0&\sqrt{2}
\end{bmatrix}.
\end{align*}
Taking determinant we obtain $ \mathrm{NJ}(\gamma, (\mathrm 1_3,\mathrm 1_3)) = \sqrt{\det JJ^T} = \frac{1}{\sqrt{8}}.$ We get $\mathrm{vol}(\mathcal E_{\S}) = 4\pi^3.$ As above, this implies $\mathrm{vol}(\mathcal E) = 4 \cdot \mathrm{vol}(\realproj^5)$.
\end{proof}

Another important notion in the context of relative pose problems in computer vision is the so-called \emph{fundamental matrix}; see, e.g., \cite[Section 9]{hartley_zisserman_2004}. While essential matrices encode the relative pose of calibrated cameras, fundamental matrices encode the relative position between uncalibrated cameras. Fundamental matrices are {precisely} the matrices of rank two. So, similar to \cref{lem:volEssential}, the average degree of {fundamental} matrices is given by the normalized volume of the manifold of rank two matrices $\mathcal F \subset \realproj^8$. The volume was computed by Beltr\'an in \cite{Beltr2009}:
$\mathrm{vol}(\mathcal F) = \frac{\pi^4}{6} =  2\cdot \mathrm{vol}(\realproj^7).$ Notice that $\dim \mathcal F =7$.
We get
$$
    \mean_{L\sim \mathrm{Unif}(\mathbb G)} \vol(\mathcal F \cap L)  = \frac{\vol(\mathcal F)}{\vol(\realproj^{7})} = 2.
$$
(here, $L=U\cdot L_0, U\sim \mathrm{Unif}(\mathrm O(9))$, is a random uniform line in $\realproj^8$).

Thus, the average degree of the manifold of fundamental matrices is 2, while the degree of its complexification is 3.

\medskip

\section{Average number of relative poses}\label{sec:proof_main2}
In this section we prove \cref{main2}. 
Let $\Psi:(\realproj^2)^{\times 10}\to\R$ be a measurable function and denote  
$\bfp := (\bfu_1,\bfv_1,\ldots,\bfu_5,\bfv_5)\in(\realproj^2)^{\times 10},$ {where $(\realproj^2)^{\times 10}$, represents taking the cartesian product of $(\realproj^2)$ with itself $10$ times.} 
We consider the {following} expected value {for the number of real solutions to the relative pose problem}
$$\mu:=\mean\limits_{\bfu_1,\bfv_1,\ldots, \bfu_5,\bfv_5\sim \mathrm{Unif}(\realproj^2) \text{ i.i.d.}} \,\Psi(\bfp) \cdot \#\{E\in\mathcal E \mid \bfu_{1}^T E \bfv_{1}=\cdots =\bfu_{5}^T E \bfv_{5}=0\}.$$ For $\Psi(\bfp)=1$, the constant one function, $\mu=\mean_{L\sim \psi} \# (\mathcal E\cap L)$. In the general case, $\mu$ is the expected value of $\# (\mathcal E\cap L)$ for a probability distribution with probability density $\Psi(\bfp)$.

\smallskip

\begin{example}\label{ex1}
We regard $\R^2$ as a subset of $\realproj^2$ by using the embedding $\phi:\R^2\to\realproj^2$ such that $ \bfu:=\phi(\bfy)=[\bfy : 1]$.
Consider the case when $\bfy\in\R^2$ is chosen uniformly in the box $B:=[a,b]\times[c,d]\subset \R^2$. We compute the probability density of $\bfu$ relative to the uniform measure on $\realproj^2$.
The probability density of $\bfy$ relative to the Lebesgue measure in $\mathbb{R}^2$ is $\frac{1}{(b-a)(d-c)}\cdot \delta_B(\bfy)$, where $\delta_B(\bfy)$ is the indicator function of the box $B$. 
Let $W\subset \phi(B)$ be a measurable subset, then ${P}(\bfu\in W)={P}(\bfy\in \phi^{-1}(W))=\int_{\phi^{-1}(W)}\frac{1}{(b-a)(d-c)}\; \mathrm{d}\bfy $. Using the coarea formula \cref{coarea_formula} we express the probability of $W$ as
$$
{P}(\bfu\in W) =\int_{W} \frac{1}{(b-a)(d-c)}\cdot \frac{1}{\mathrm{NJ}(\phi,\bfy)} \; \mathrm{d}\bfu.
$$
Therefore, the probability density of $\bfu$ is $\left((b-a)(d-c)\cdot \mathrm{NJ}(\phi,\bfy)\right)^{-1}$.
Let us compute the normal Jacobian of the map $\phi$. Since we can work locally, we compute the derivative of the map $ \bfy \mapsto \bfs:= \frac{1}{\sqrt{y_1^2+y_2^2+1}}(y_1,y_2,1) \in \S^2$. The derivative of this map relative to the standard basis in $\mathbb{R}^2$ and $\mathbb{R}^3$ is expressed by the matrix $$
\frac{1}{\sqrt{y_1^2+y_2^2+1}} \begin{bmatrix}
1 & 0  \\
0 & 1\\
0 & 0
\end{bmatrix} + \left(\frac{\partial}{\partial y_1} \frac{1}{\sqrt{y_1^2+y_2^2+1}}  \right)\begin{bmatrix}
y_1 & 0 \\
y_2 &0 \\
1 & 0
\end{bmatrix} + \left(\frac{\partial}{\partial y_2} \frac{1}{\sqrt{y_1^2+y_2^2+1}}  \right)\begin{bmatrix}
0 & y_1 \\
0& y_2  \\
0 &1 
\end{bmatrix}. 
$$
The tangent space of the sphere is $T_{\bfs}\S^2=\bfs^\perp$. Let $P_{\bfs}=\mathrm{1}_3-\bfs \bfs^T$ be the projection onto $\bfs^\perp$. To get the derivative relative to an orthonormal basis of $\bfs^\perp$, we have to multiply the above matrix from the left with $P_{\bfs}$. We get
$$
\mathrm{NJ}(\phi,\bfy) = \frac{1}{y_1^2+y_2^2+1} \cdot \sqrt{\det M^TM}, \; \text{ where } M= P_{\bfs}  \begin{bmatrix}
1 & 0  \\
0 & 1\\
0 & 0
\end{bmatrix}.
$$
We have $\sqrt{\det M^TM}=\vert \langle \bfs , \bfe_3 \rangle \vert $. This implies that the probability density of $\bfu$ is given by $$
\frac{1}{(b-a)(d-c)} \cdot \frac{1}{\mathrm{NJ}(\phi,\bfy) }=\frac{1}{(b-a)(d-c)}\cdot \frac{u_1^2+u_2^2+u_3^2}{u_3^2}  \cdot \frac{1}{\cos{\alpha}},
$$
where $\alpha$ is the angle between the lines through $\bfu$ and $\bfe_3$. 
 
Let us write $g(\bfu):=\frac{u_1^2+u_2^2+u_3^2}{u_3^2}  \cdot \frac{1}{\cos{\alpha}}$. If for $1\leq i\leq 5$ we choose independently~$\bfu_i$ from the box $[a_i,b_i]\times[c_i,d_i]$ and $\bfv_i$ from the box $[a_i',b_i']\times[c_i',d_i']$ we obtain the density $\Psi(\bfp)$ with
 $$\Psi(\bfp) = \frac{g(\bfu_1)}{(b_1-a_1)(d_1-c_1)}\cdots \frac{g(\bfu_5)}{(b_5-a_5)(d_5-c_5)}\cdot \frac{g(\bfv_1)}{(b_1'-a_1')(d_1'-c_1')}\cdots \frac{g(\bfv_5)}{(b_5'-a_5')(d_5'-c_5')},$$
 when $\bfp$ is in the product of boxes, and $\Psi(\bfp)=0$ otherwise.
 \exampleend
\end{example}   

\bigskip

We will also denote $\Psi:(\R^3\setminus\{0\})^{\times 10}\to\R$ defined by $\Psi(\bfu_1,\ldots,\bfv_5):=\Psi(\pi(\bfu_1),\ldots,\pi(\bfv_5)),$ where $\pi:\R^3\setminus\{0\}\to\realproj^2$ is the projection.
It will be convenient to replace the uniform random variables in $\realproj^2$ by Gaussian random variables in $\mathbb R^3$, see \cite[Remark 2.24]{condition}:
\begin{equation}\label{eq4}
\mu = \mean\limits_{\bfu_1,\bfv_1,\ldots, \bfu_5,\bfv_5\sim N(0,\mathrm{1}) \text{ i.i.d.} } \,\Psi(\bfp)\cdot \#\{E\in\mathcal E \mid \bfu_{1}^T E \bfv_{1}=\cdots =\bfu_{5}^T E \bfv_{5}=0\}.
\end{equation}
Again, $\mean_{L\sim \psi} \# (\mathcal E\cap L)$ is recovered by setting $\Psi(\bfp)=1$ in \cref{eq4}.
We denote the Gaussian density 
by $\Phi(\bfp)=(2\pi)^{-15} \exp(-\tfrac{1}{2} \sum_{i=1}^5\Vert \bfu_i\Vert^2 + \Vert \bfv_i\Vert^2 )$. 

The proof of \cref{main2} consists of three steps, separated into three subsections. {In the initial two subsections, our objective is to calculate the normal Jacobian and apply the coarea formula. However, in this process, we do not arrive at an explicit or practical form. Following that in the final subsection, we adopt an alternative approach that involves a new parametrization. This transformation allows us to obtain a closed-form expression for \cref{main2}.}

\subsection{The incidence variety} 
The incidence variety is
$$\mathcal I := \{(\bfp,E)\in  ({\mathbb R^3\setminus \{0\}})^{\times 10}\times\mathcal E \mid \bfu_{1}^T E \bfv_{1}=\cdots =\bfu_{5}^T E \bfv_{5}=0\}.$$
This is a real algebraic subvariety of $({{\mathbb R^3\setminus \{0\}}})^{\times 10}\times\mathcal E$.  Recall from \cref{lemma:invariance_phi} that $\mathrm{SO}(3)\times \mathrm{SO}(3)$ acts transitively on $\mathcal E$ by left and right multiplication. This extends to a group action on $\mathcal I$ via 
$(U,V).(\bfp,E) := (U\bfu_1,V\bfv_1,\ldots,U\bfu_5,V\bfv_5,\ UEV^T).$
Let $E_0:=E(\mathrm 1_3, \bfe_1)$ be as in \cref{def_E_0} and let us denote the quadric 
$$
q(\bfu,\bfv):=\bfu^T E_0\bfv =  \bfu^T\begin{bmatrix}0&0&0\\ 0&0&-1\\0&1&0\end{bmatrix}\bfv=-\det \begin{bmatrix} u_2 & u_3\\ v_2&v_3\end{bmatrix},
$$
where $\bfu=(u_1,u_2,u_3)^T$ and $\bfv=(v_1,v_2,v_3)^T$. We denote its zero set by 
$$Q = \{(\bfu,\bfv)\in\mathbb R^3 \times \mathbb R^3 \mid q(\bfu,\bfv) = 0\}.
$$
Since $\mathcal E$ is an orbit of the $\mathrm{SO}(3)\times \mathrm{SO}(3)$ action, 
$\mathcal I = \bigcup_{(U,V) \in \mathrm{SO}(3)\times \mathrm{SO}(3)} \; (U,V).(Q^{\times 5}\times \{E_0\}).$
Let us denote $\widetilde{Q}:=\{(\bfu,\bfv)\in Q\mid \bfu,\bfv\not \in \mathbb R\bfe_1\}$. This is a Zariski open subset of $Q$. Let
$$\widetilde{\mathcal I} := \bigcup_{(U,V) \in \mathrm{SO}(3)\times \mathrm{SO}(3)} \; (U,V).(\widetilde{Q}^{\times 5}\times \{E_0\}).$$
We prove that $\widetilde{\mathcal I}$
is smooth by showing that the Jacobian matrix of the system of equations $\bfu_i^T{E}\bfv_i=0,$ for $i=1,\ldots,5$ has full rank at every point in $\widetilde{\mathcal I}$; see, e.g., \cite[Theorem A.9]{condition}.

The Jacobian matrix of $q$ is the $1\times 6$ matrix
$J(\bfu,\bfv) := \begin{bmatrix} 0 & -v_3 & v_2& 0 & u_3 & -u_2\end{bmatrix}$. Denote
\begin{equation}\label{def_A}
A:=\begin{bmatrix}
J(\bfu_1,\bfv_1) & & & & \\
 & J(\bfu_2,\bfv_2)  & & & \\
 & & J(\bfu_3,\bfv_3)  & & \\
 & & & J(\bfu_4,\bfv_4)  & \\
 & & & & J(\bfu_5,\bfv_5)
\end{bmatrix}\in\mathbb R^{5\times 30}.
\end{equation}
For $(\bfp, E_0)\in  \widetilde{\mathcal I}$ the matrix $A$ has full rank. Since the image of $A$ is contained in the image of the Jacobian matrix of $\bfu_i^T{E}\bfv_i=0, i=1,\ldots,5$, we see that the latter has full rank. Therefore, $\widetilde{\mathcal I}$ is smooth.

\subsection{Computing the normal Jacobian}

On $\mathcal I$ we have the two coordinate projections
$\pi_1: \mathcal I \to (\mathbb R^3\setminus\{0\})^{\times 10}$ and $\pi_2:\mathcal I\to \mathcal E$.
The projection $\pi_2$ is surjective, but $\pi_1$ is not, since out of the 10 complex solutions of the system of equations $\bfu_i^T{E}\bfv_i=0, i=1,\ldots,5$, 
there can be~0 real solutions. 
{Notice that $\operatorname{im}(\pi_1)$ is a full-dimensional semi-algebraic set.  
Let $\mathcal U$ be the interior of $\operatorname{im}(\pi_1)$. Then, $\mathcal U$ is an open set, hence measurable. Integrating over $\operatorname{im}(\pi_1)$ is the same as integrating over $\mathcal U$. 
  We therefore have, using \cref{eq4},}
\begin{equation}\label{eq5}
    \mu = \int_{\mathcal U} {\Psi(\bfp)}\, \#{\pi_1}^{-1}(\bfp)\;\Phi(\bfp)\;\mathrm d \bfp.
\end{equation}
Let us also denote $\widetilde{\mathcal U}:=\pi_1(\widetilde{\mathcal I})$. Consider a point $\bfp\in\mathcal U\setminus \widetilde{\mathcal U}$ and suppose that $(\bfp, E)\in\mathcal I$. Let $(U,V)\in \mathrm{SO}(3)\times \mathrm{SO}(3)$ such that $(U,V).E=E_0$. Since $\widetilde{Q}$ is Zariski open in $Q$, every neighborhood of~$(U,V).\bfp$ intersects $\widetilde{Q}$. Consequently, every neighborhood of $\bfp$ intersects $\widetilde{\mathcal U}$. This means that $\mathcal U'$ is open dense in $\mathcal U$ in the Euclidean topology. Hence, in \cref{eq5} we can replace $\mathcal U$ by $\widetilde{\mathcal U}$ to get
$$ \mu = \int_{\widetilde{\mathcal U}}  \Phi(\bfp)\cdot\Psi(\bfp)\cdot \#{\pi_1}^{-1}(\bfp)\;\mathrm d \bfp.$$
We have shown in the previous subsection that $\widetilde{\mathcal I}$ is a smooth manifold. We may therefore apply the coarea formula from \cref{coarea_formula} twice, first to $\pi_1$ and then to $\pi_2$, to get
\begin{align*}
    \mu &=\int_{\widetilde{\mathcal I}}\,\Phi(\bfp)\cdot\Psi(\bfp)\cdot \mathrm{NJ}(\pi_1,(\bfp, E))\; \mathrm d(\bfp, E)\\
&= \int_{\mathcal E}
\left(\int_{\pi^{-1}_{2}(E)}\, \Phi(\bfp)\cdot\Psi(\bfp)\cdot\frac{\mathrm{NJ}(\pi_1,(\bfp, E))}{\mathrm{NJ}(\pi_2,(\bfp,E))}\; \mathrm d\bfp\right)\;\mathrm d E.
\end{align*}
Let now $(U,V)\in\mathrm{SO}(3)\times\mathrm{SO}(3)$ such that $UEV^T = E_0$. It follows from \cref{lemma:invariance_phi} that $\pi_1,\pi_2$ are equivariant, which implies that $\mathrm{NJ}(\pi_{i},(\bfp,E)) = \mathrm{NJ}(\pi_{i}, (U,V).(\bfp,E)), i=1,2$. Furthermore, the Gaussian density $\Phi(\bfp)$ is also invariant under the $\mathrm{SO}(3)\times\mathrm{SO}(3)$ action. The fiber over~$E_0$ is
$\pi_2^{-1}(E_0) =  \widetilde{Q}^{\times 5}\times \{E_0\}$, which is open dense in $Q^{\times 5}\times \{E_0\}$. So,
\begin{equation}\label{eq1}
    \mu = \int_{\mathcal E}
    \left(\int_{Q^{\times 5}}\, \Phi(\bfp)\cdot\Psi((U,V).\bfp)\cdot\frac{\mathrm{NJ}(\pi_1,(\bfp, E_0))}{\mathrm{NJ}(\pi_2,(\bfp,E_0))}\; \mathrm d\bfp\right)\;\mathrm d E,
\end{equation}
where $(U,V)\in\mathrm{SO}(3)\times \mathrm{SO}(3)$ is such that $E=(U,V).E_0$.
The ratio of normal Jacobians is computed next. 

Recall from \cref{def_A} the definition of the matrix $A\in\R^{5\times 30}$.
For $B_1,\ldots,B_5$ the basis from \cref{prop_TS} we denote
$$B :=
\begin{bmatrix}
    \bfu_1^TB_1\bfv_1 & \cdots & \bfu_1^TB_5\bfv_1\\
    \vdots & \ddots & \vdots\\
    \bfu_5^TB_1\bfv_5 & \cdots & \bfu_5^TB_5\bfv_5
\end{bmatrix}
\in\mathbb R^{5\times 5}.
$$
Then, the tangent space of $\widetilde{\mathcal I}$ at $(\bfp, E_0)$ is defined by the linear equation
$A\dot \bfp + B\dot E=0$. 
Therefore, when~$B$ is invertible, $-B^{-1}A$ is a matrix representation for $\mathrm D_{(\bfp,E_0)}\pi_2\mathrm D_{(\bfp,E_0)}\pi_1^{-1}$ with respect to orthonormal bases. So,
\begin{equation}\label{eq2}
    \frac{\mathrm{NJ}(\pi_1,(\bfp,E_0))}{\mathrm{NJ}(\pi_2,(\bfp,E_0))} = \frac{1}{\vert\det(B^{-1}AA^TB^{-T})\vert} = \frac{\vert \det(B)\vert}{\sqrt{\det(AA^T)}}.
\end{equation}
When $B$ is not invertible, $\mathrm{NJ}(\pi_1,(\bfp,E_0))=0$ and the formula in \cref{eq2} also holds. 

\subsection{Integration on the quadric} 
We plug \cref{eq2} into \cref{eq1} and obtain
$$\mu 
=\int_{\mathcal E}
\left(\int_{Q^{\times 5}}\, \Phi(\bfp)\cdot\Psi((U,V).\bfp)\cdot\frac{\vert \det(B)\vert}{\sqrt{\det(AA^T)}}\; \mathrm d\bfp\right)\;\mathrm d E.$$
We denote $f(\bfu,\bfv):=u_2^2 + u_3^2 + v_2^2 + v_3^2$ for $\bfu=(u_1,u_2,u_3), \bfv=(v_1,v_2,v_3)$. Then,
$$\det(AA^T) = \prod_{i=1}^5 f(\bfu_i,\bfv_i).$$
We have $(\bfu,\bfv)\in Q$ if and only if $(u_2,u_3)$ is a multiple of $(v_2,v_3)$. Therefore, we have the following $2:1$ parametrization:
\begin{align*}
&\phi: \mathbb R^4\times [0,2\pi)  \to Q,\; (a,b,r,s,\theta)\mapsto (\bfu,\bfv),\\
&\text{ where } \bfu=(a, r\cdot \bfw)^T,\quad \bfv=(b, s\cdot \bfw)^T,\quad \bfw=(\cos\theta, \sin\theta)\in\S^1.
\end{align*}
The Jacobian matrix of $\phi$ is
$$J = \begin{bmatrix}
1&0 & 0&0&0\\
 0 &0& \cos\theta& 0&-r\sin\theta  \\
  0 &0& \sin\theta& 0&r\cos\theta  \\
  0&1& 0&0&0\\
 0 &0& 0 & \cos\theta& -t\sin\theta  \\
  0 &0& 0 & \sin\theta &t\cos\theta
\end{bmatrix}\in\mathbb R^{6\times 5}.$$
Then, $\mathrm{NJ}(\phi, (a,b,r,s,\theta)) = \sqrt{\det(J^TJ)}$ and
$$\det(J^TJ) = r^2 + s^2= u_2^2 + u_3^2 + v_2^2 + v_3^2 = f(\bfu,\bfv).$$ 
Let us denote $\bfa:=(a_i,b_i,r_i,s_i,\theta_i)_{i=1}^5$. We get:
\begin{equation}\label{eq6}
    \mu   = \frac{1}{2^5}\int_{\mathcal E}
    \left(\int_{(\mathbb R^4\times [0,2\pi))^{\times 5}}\, \Phi(\phi(\bfa))\cdot\Psi((U,V).\phi(\bfa))\cdot \vert \det(B)\vert\; \mathrm d\bfa\right)\;\mathrm d E.
\end{equation}
Notice that
$\Phi(\phi(\bfa)) = \tfrac{1}{(2\pi)^{5}}\,\tfrac{1}{(2\pi)^{10}}\,\exp(-\tfrac{1}{2}\sum_{i=1}^5 (a_i^2 + b_i^2 + r_i^2 + s_i^2))$ is the probability density, such that $a_i,b_i,r_i,s_i$ are all standard normal and $\theta_i$ is uniform in $[0,2\pi)$ for every $i$, and all variables are independent.
We can therefore rephrase \cref{eq6} as
$$\mu = \frac{1}{2^5}\int_{\mathcal E} \bigg(\mean_{a_i,b_i,r_i,s_i\sim N(0,1)}\; \mean_{\theta_i\sim \mathrm{Unif}([0,2\pi)), i=1,\ldots,5} \Psi((U,V).\phi(\bfa))\cdot \vert\det(B)\vert\bigg)\;\mathrm d E.$$
The rows of $B$ are all of the form
$$
\begin{bmatrix}
    \bfu^TB_1\bfv\\
    \bfu^TB_2\bfv\\
    \bfu^TB_3\bfv\\
    \bfu^TB_4\bfv\\
    \bfu^TB_5\bfv
\end{bmatrix} = \begin{bmatrix}
    \sqrt{2}\, u_3v_1\\
    \sqrt{2}\, u_2v_1\\
    \sqrt{2}\, u_1v_3\\
    \sqrt{2}\, u_1v_2\\
    \ u_2v_2+u_3v_3\
\end{bmatrix} = \begin{bmatrix}
    \sqrt{2}\cdot b\cdot r\cdot \sin\theta \\
    \sqrt{2}\cdot b\cdot r\cdot \cos \theta\\
    \sqrt{2}\cdot a \cdot s \cdot \sin\theta\\
    \sqrt{2}\cdot a \cdot s \cdot \cos\theta\\
    rs
\end{bmatrix}.
$$
This shows that $\vert \det(B)\vert \sim 4 \cdot \vert \det \begin{bmatrix} \bfz_1 & \ldots & \bfz_5\end{bmatrix}\vert$, where $\bfz_1,\ldots, \bfz_5\sim \bfz$ i.i.d.\ for
\begin{equation}
\label{def_z}
\bfz=
\begin{bmatrix}
b\cdot r\cdot \sin\theta \\
b\cdot r\cdot \cos \theta\\
a \cdot s \cdot \sin\theta\\
a \cdot s \cdot \cos\theta\\
rs
\end{bmatrix}, \quad a,b,r,s\sim N(0,1), \quad \theta\sim \mathrm{Unif}([0,2\pi)),\quad  \text{all independent}.
\end{equation}
We state a general integral formula.
\begin{theorem}\label{main3}
{With the notation above, we have that the expected value $\mu = \mean\#(\mathcal E\cap L)$, where the distribution of $L$ is defined by a nonnegative measurable function $\Psi:(\realproj^2)^{\times 10}\to \R$, is given by
$$\mu = \frac{\mathrm{vol}(\mathcal E)}{2^3}\mean_{E\in \mathcal E} \, \mean_{\bfa}\; \Psi((U,V).\phi(\bfa))\cdot \vert\det\begin{bmatrix} \bfz_1 & \ldots & \bfz_5\end{bmatrix}\vert,$$
where $(U,V)\in\mathrm{SO}(3)\times\mathrm{SO}(3)$ is such that $E=(U,V).E_0$, and the first expected value is over the uniform distribution in $\mathcal E$. The second expectedd value is for $\bfa=(a_i,b_i,r_i,s_i,\theta_i)_{i=1}^5$ with $a_i,b_i,r_i,s_i\sim N(0,1),\, \theta_i\sim \mathrm{Unif}([0,2\pi))$ and all are independent.}
\end{theorem}

{We continue \cref{ex1} by computing the distribution and approximating the mean value.}
\begin{example}\label{ex11}
{As in \cref{ex1} we consider the case when the $\bfx_i$ and $\bfy_i$ are sampled i.i.d.\ from the box $[-5,5]\times [-5,5]\subset \mathbb R^2$. Figure \ref{fig:experiment3} shows the empirical distribution of the number of real zeros and an empirical mean of $\approx 3.788$. We could also approximate the average number of real zeros using \cref{main3}.} 

{We sample from the probability density $\Psi((U,V).\phi(\bfa))$ in \cref{main3} using the basic version of Metropolis-Hastings algorithm (see, e.g.,~\cite{Roberts2004}). For this, we use the proposal density for $(E,\bfa)$, such that $\bfa$ is as above and $E\in \mathcal E$ is uniform. We computed a corresponding Markov-Chain with $10^6$ states. The Metropolis-Hastings algorithm rejected all but 796 of those states. 
The empirical mean computed from the 796 states is~$\approx 3.5563$.
}
\begin{figure}
    \centering
    \includegraphics[width = 0.49\textwidth]{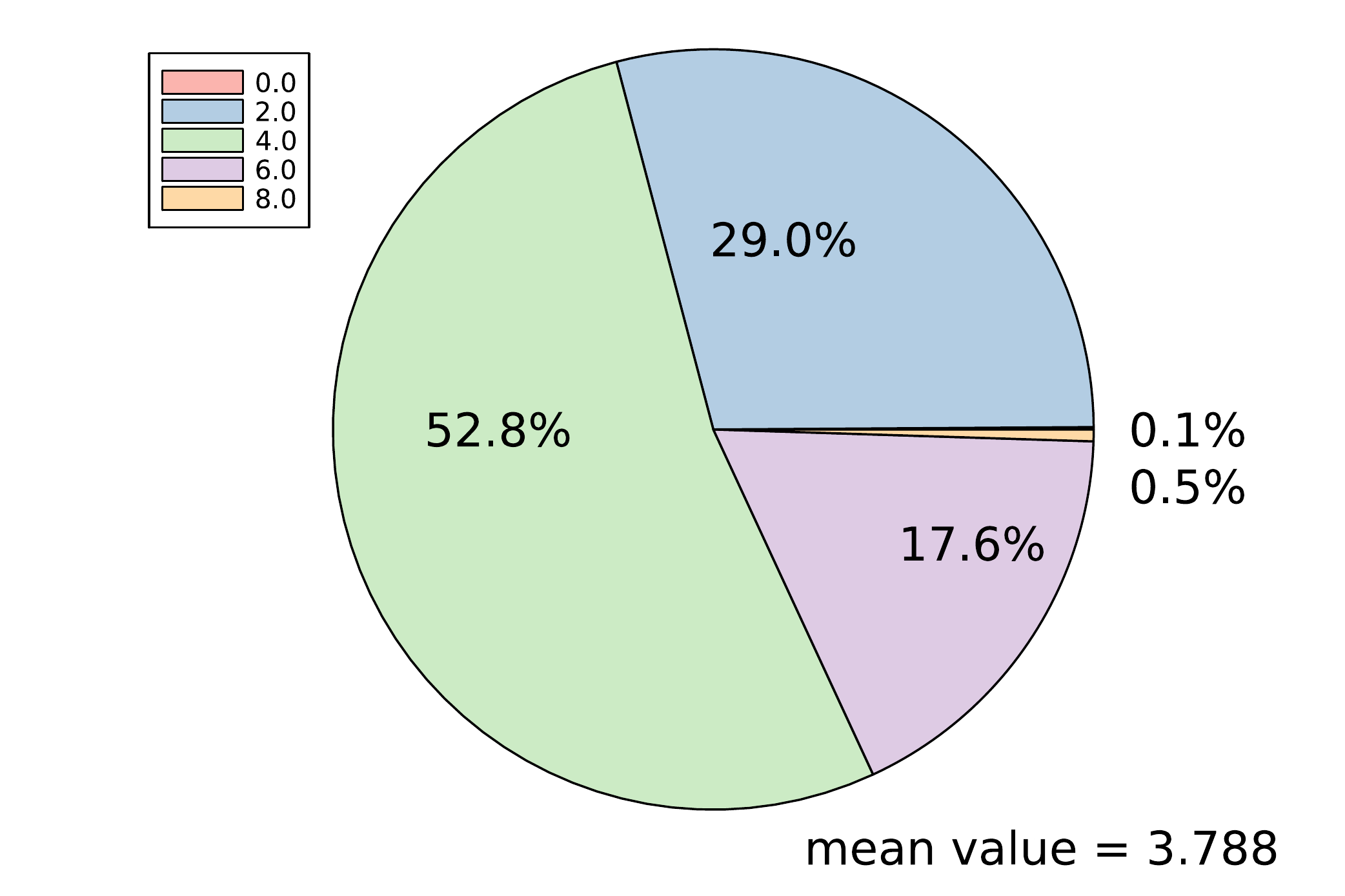}
    \caption{
    {The pie chart shows the outcome of the following experiment, similar to the experiments in Figure \ref{fig:exp}. We sampled $N=1000$ pairs $(\bfx_i,\bfy_i)_{i=1}^5$, where the $\bfx_i$ and $\bfy_i$ are sampled i.i.d.\ from the box $[-5,5]\times [-5,5]\subset \mathbb R^2$. Then, we computed $\mathcal E\cap L$ by solving the system of polynomial equations with the software \texttt{HomotopyContinuation.jl} \cite{HC.jl}. The chart shows the empirical distribution of real zeros and the corresponding empirical means in these experiments.}}
    \label{fig:experiment3}
\end{figure}
    
\end{example}

Let us now work towards proving \cref{main2}. In the setting of \cref{main2} we have $\Psi(\bfp)=1$ and thus, by \cref{main3},
$
\mean_{L\sim \psi} \# (\mathcal E\cap L)    = 2^{-3}\cdot \mathrm{vol}(\mathcal E)\cdot \mean  \left\vert\det \begin{bmatrix} \bfz_1 & \bfz_2&\bfz_3&\bfz_4 & \bfz_5\end{bmatrix}\right\vert.
$
We have shown in \cref{lem:volEssential} that $\vol(\mathcal E) = 4\cdot \mathrm{vol}(\realproj^5) = 2
\pi^3$. Consequently,
$$\mean_{L\sim \psi} \# (\mathcal E\cap L)  = \frac{\pi^3}{4} \cdot \mean  \left\vert\det \begin{bmatrix} \bfz_1 & \bfz_2&\bfz_3&\bfz_4 & \bfz_5\end{bmatrix}\right\vert$$
as stated in \cref{main2}. 

We close this section by giving a (extremely coarse) upper bound on the variance of the random determinant. This bound is used for applying Chebychev's inequality in the introduction.
\begin{proposition}\label{prop_var}
$\mathrm{Var}\left(\ \left\vert\det \begin{bmatrix} \bfz_1 & \bfz_2&\bfz_3&\bfz_4 & \bfz_5\end{bmatrix}\right\vert\ \right)\leq 360$.
\end{proposition}
\begin{proof}
Let $D$ denote the random absolute determinant. We have $\mathrm{Var}(D)\leq \mean D^2$. Expanding the determinant with Laplace expansion, multiplying out the square, and taking the expected value we see that all mixed terms (that is, all terms which are not a square) average to 0 because the distributions of $a,b,r,s$ are symmetric around 0. This implies
$$\mean D^2 = 5!\cdot {\mean \left[(br\sin\theta)^2 + (br\cos\theta)^2 + (as\sin\theta)^2 + (as\cos\theta)^2 + (rs)^2 \right]}= 5! \cdot 3 = 360,$$
where we have used that $\mean_\theta \cos^2\theta = \mean_\theta \sin^2\theta = \tfrac{1}{2}$.
\end{proof}

\medskip

\section{The essential zonoid}\label{sec_zonoid}
Vitale \cite{Vitale} showed that the expected absolute determinant of a random matrix can be expressed as the volume of a convex body. More specifically, of a \emph{zonoid}. Zonoids are limits of zonotopes in the Hausdorff topology on the space of all convex bodies, and zonotopes are Minkowski sums of line segments; see \cite{schneider14} for more details.

Notice that the probability distribution of $\bfz$ from \cref{def_z} is invariant under multiplying by~$-1$; i.e., $\bfz\sim -\bfz$. In this case, based on Vitale's result, it was shown in \cite[Theorem 5.4]{BBLM2020} that~$\mean  \left\vert\det \begin{bmatrix} \bfz_1 & \bfz_2&\bfz_3&\bfz_4 & \bfz_5\end{bmatrix}\right\vert = 5!\cdot \mathrm{vol}(K)$, where $K\subset\mathbb R^5$ is the convex body with support function $h_K(\bfx) = \tfrac{1}{2}\mean \vert \langle \bfx,\bfz\rangle \vert$. So
\begin{equation}
\label{eq9}\mean_{L\sim \psi} \# (\mathcal E\cap L)  = 5!\cdot \frac{\pi^3}{4}\cdot \mathrm{vol}(K).
\end{equation}
We call $K$ the  \emph{essential zonoid}. 

In the remainder of this section, we bound $h_K(\bfx)$ from below to find a convex body whose volumes give a lower bound for~$\mathrm{vol}(K)$. This gives, using \cref{eq9}, the following result.
\begin{proposition}\label{thm: lower bound}
$\displaystyle \mean_{L\sim \psi} \# (\mathcal E\cap L) \geq  0.93$.
\end{proposition}
{It is important to note that the mentioned lower bound involves numerical computations in its calculation.}
\begin{remark}
The value of $0.93$ is not close to the experimental value of $3.95$ from the introduction. To get a lower bound closer to $3.95$ one would need to understand the support function of $K$ at points $\bfx=(x_1,\ldots,x_5)\in\mathbb R^5$, where all entries are nonzero. In the computation below we always either have $x_1=x_2=0$ or $x_3=x_4=0$. For such points we can work with the function that maps $\bfx$ to the vector of norms $\boldsymbol \rho=(\rho_1,\rho_2,\rho_3)$, where~$\rho_1=\sqrt{x_1^2+x_2^2}, \rho_2 = \sqrt{x_3^2+x_4^2}$ and $\rho_3 = \vert x_5\vert$. However, if all entries of $\bfx$ are nonzero, also the angle between the two points~$(x_1,x_2),(x_3,x_4)\in\mathbb R^2$ will play a role, not just their norms. We were not able to find a lower bound for $h_K(\bfx)$ in this case. We nevertheless prove \cref{thm: lower bound} for completeness.
\end{remark}

We will need the following lemma.
\begin{lemma}\label{lem_expected_values}
We have 
\begin{enumerate}
\item $\displaystyle\mean_{\xi\sim N(0,\sigma^2)} \vert\xi\vert =\sigma \sqrt{\tfrac{2}{\pi}}$;
\item  
$\displaystyle\mean_{\theta\sim \mathrm{Unif}([0,2\pi))} \vert\cos\theta\vert = \tfrac{2}{\pi}.$
\end{enumerate}
\end{lemma}
\begin{proof}
The first formula is proved by using 
$\mean_{\xi\sim N(0,1)} \vert\xi\vert = 2\int_{0}^\infty x \cdot  \tfrac{1}{\sqrt{2\pi}}e^{-\frac{1}{2}x^2}\;\mathrm d x = \sqrt{\tfrac{2}{\pi}},$ and then $\mean_{\xi\sim N(0,\sigma^2)} \vert\xi\vert =\sigma \mean_{\xi\sim N(0,1)} \vert\xi\vert $.
The second is
$\mean \vert\cos\theta\vert = 4\int_0^{\frac{\pi}{2}} \cos\theta \cdot \tfrac{1}{2\pi}\; \mathrm d \theta = \tfrac{2}{\pi}.$
\end{proof}

Let us have a closer look at the support function. 
\begin{align*}  
h_K(\bfx) 
&= \frac{1}{2}\mean \vert\langle \bfx, \bfz\rangle \vert\\[0.5em]
&= \frac{1}{2}\mean \vert br(x_1 \sin\theta + x_2 \cos\theta) + as(x_3 \sin\theta + x_4 \cos\theta) + x_5rs\vert\\[0.5em]
&= \frac{1}{2}\mean \left\vert
\begin{bmatrix} a & r \end{bmatrix} \, C \, \begin{bmatrix} b \\s \end{bmatrix}\right\vert,\end{align*}
where $C$ is the $2\times 2$ matrix
\[C := \begin{bmatrix}
    0 & x_3 \sin\theta + x_4 \cos\theta \\
    x_1 \sin\theta + x_2 \cos\theta & x_5
    \end{bmatrix}.
\]
Let $\sigma_1\geq\sigma_2\geq 0$ denote the two singular values of $C$. The Gaussian vectors $(a,r)$ and $(b,s)$ are invariant under rotations. Therefore, $h_K(\bfx) = \frac{1}{2}\mean \vert \sigma_1ab + \sigma_2rs\vert$. The law of adding Gaussians implies that for fixed $a,r$ and random $b,s$ we have $\sigma_1 ab + \sigma_2 rs\sim N(0, \sigma_1^2a^2 + \sigma_2^2r^2)$. We now keep $a,r$ fixed and take the expectation with respect to~$b,s$. This gives, using the first formula from \cref{lem_expected_values},
\begin{equation}\label{support_fct}
h_K(\bfx)  = \frac{1}{\sqrt{2\pi}}\,\mean_{a,r,\theta} \sqrt{\sigma_1^2a^2 + \sigma_2^2r^2}.\end{equation}

For $\bfx\in\R^5$ let us write 
$$\rho_1:=\sqrt{x_1^2+x_2^2}, \quad \rho_2:=\sqrt{x_3^2+x_4^2}\quad \text{ and }\quad \rho_3:=\vert x_5\vert.$$
From \cref{support_fct} we have $h_K(\bfx)   \geq \frac{1}{\sqrt{2\pi}}\,\mean_{a,\theta} \vert\sigma_1a\vert$ as $\sigma_2^2r^2\geq 0$. Since $\sigma_1$ does not depend on $a$ and $a,\theta$ are independent, this gives $h_K(\bfx)\geq \frac{1}{\sqrt{2\pi}}\,\mean_{a} \vert a\vert\mean_\theta \vert\sigma_1\vert$. Using \cref{lem_expected_values} we get
$$h_K(\bfx) \geq \frac{1}{\pi}\,  \mean_{\theta} \sigma_1.$$
The larger singular value $\sigma_1$ can be expressed as
$$\sigma_1 = \max_{\bfa,\bfb\in\mathbb R^2: \Vert\bfa\Vert = \Vert\bfb\Vert = 1} \bfa^T\, C\, \bfb.$$
Therefore,
$$h_K(\bfx) \geq \frac{1}{\pi}\, \mean_\theta \vert\bfe_2^T\,C\,\bfe_1\vert = \frac{1}{\pi}\, \mean_\theta \vert x_1\sin\theta + x_2\cos\theta\vert = \frac{2}{\pi^2}\cdot\rho_1;$$
the last equality by rotational invariance and \cref{lem_expected_values}. Similarly,
$h_K(\bfx) \geq \tfrac{2}{\pi^2}\rho_2,$ and also~$h_K(\bfx) \geq \tfrac{1}{\pi} \rho_3$.

We recall the definition of the \emph{elliptic integral of the second kind}
$$\mathrm E(m):=\int_0^{\frac{\pi}{2}} \sqrt{1-m\sin^2\theta}\;\mathrm d\theta {,}$$
and define
$$F(x,y) := \tfrac{2}{\pi^2}\cdot \sqrt{x^2+y^2}\cdot \mathrm E\left(\tfrac{x^2}{x^2 +y^2}\right).$$
Then, we have
\begin{align*}h_K(\bfx) \geq \frac{1}{\pi}\, \mean_\theta \Vert {C}^T\bfe_2\Vert
    &= \frac{1}{\pi}\, 
    \mean_\theta \sqrt{(x_1\sin\theta + x_2\cos \theta)^2 + x_5^2}\\ &= \frac{1}{\pi}\, \mean_\theta\sqrt{\rho_1^2\cos^2\theta  + \rho_3^2}\\
    &=\frac{1}{2\pi^2} \int_0^{2\pi} \sqrt{\rho_1^2\cos^2\theta  + \rho_3^2}\;\mathrm d\theta \\
    &=\frac{2}{\pi^2}\cdot \int_0^{\frac{\pi}{2}} \sqrt{\rho_1^2\cos^2\theta  + \rho_3^2}\;\mathrm d\theta = 
    F(\rho_1,\rho_3){.}
\end{align*} 
Similarly, we have
$h_K(\bfx) \geq F(\rho_2,\rho_3).$ 

\begin{figure}
\begin{tikzpicture}[scale=3]
\coordinate (a1) at (0,0,0);
\coordinate (a2) at (1,0,0);
\coordinate (a3) at (0,1,0);
\coordinate (a4) at (0,0,1);
\coordinate (a5) at (0.73, 0, 0.73);
\coordinate (a6) at  (0, 0.72, 0.72);
\coordinate (a7) at  (0.86, 0, 2*0.86/3);
\coordinate (a8) at  (0, 0.86, 2*0.86/3);
\coordinate (a9) at  (2*0.85/3, 0, 0.85);
\coordinate (a10) at  (0, 2*0.85/3, 0.85);
\coordinate (a11) at  (0.966, 0, 0.966/3);
\coordinate (a12) at  (0, 0.966, 0.966/3);
\coordinate (a13) at  (0.957/3, 0, 0.957);
\coordinate (a14) at  (0, 0.957/3, 0.957);
 
\draw[fill=teal!20] (a4) -- (a14) -- (a10) -- (a6) -- (a8) -- (a12) --(a3) -- (a2) -- (a11) -- (a7) -- (a5) -- (a9) -- (a13) -- (a4);

\draw (a1) node[below right] {$\mathbf{0}$} node{$\bullet$};
\draw (a2) node[above right] {$\bfe_1$} node{$\bullet$};
\draw (a3) node[above right] {$\bfe_2$} node{$\bullet$};
\draw (a4) node[above left] {$\bfe_3$} node{$\bullet$};
\draw (a5) node{$\bullet$};
\draw (a6) node{$\bullet$};
\draw (a7) node{$\bullet$};
\draw (a8) node{$\bullet$};
\draw (a9) node{$\bullet$};
\draw (a10) node{$\bullet$};
\draw (a11) node{$\bullet$};
\draw (a12) node{$\bullet$};
\draw (a13) node{$\bullet$};
\draw (a14) node{$\bullet$};

\draw[dashed, thick] (a1) -- (a2);
\draw[dashed, thick] (a1) -- (a3);
\draw[dashed, thick] (a1) -- (a4); 
\draw[thick] (a4) -- (a14) -- (a10) -- (a6) -- (a8) -- (a12) --(a3);
\draw[thick] (a4) -- (a13) -- (a9) -- (a5) -- (a7) -- (a11) --(a2);
\draw[thick] (a5) -- (a6);
\draw[thick] (a2) -- (a3);
\draw[thick] (a7) -- (a8);
\draw[thick] (a9) -- (a10);
\draw[thick] (a11) -- (a12);
\draw[thick] (a13) -- (a14);

\draw[->, thick] (a2) -- (1.2,0,0);
\draw[->, thick] (a3) -- (0,1.2,0);
\draw[->, thick] (a4) -- (0,0,1.3);

\end{tikzpicture}
\caption{The polytope $P$ from \cref{def_P}.\label{fig2}}
\end{figure}

Let  $L'\subset \mathbb R^3$ be the convex body whose support function is 
$$h_{L'}(\boldsymbol\rho)=\max\left\{0,\ \tfrac{2}{\pi^2}\, \rho_1,\ \tfrac{2}{\pi^2}\, \rho_2,\ \tfrac{1}{\pi}\rho_3,\ 
F(\rho_1,\rho_3),\ F(\rho_2,\rho_3)\right\},$$
and define $\varphi: \mathbb R^5\to\mathbb R^3_{\geq 0},\; \bfx \mapsto \boldsymbol\rho$,
and
$$L:=L'\cap \mathbb R^3_{\geq 0}.$$
We have thus shown that 
$h_K(\bfx)\geq h_{\varphi^{-1}(L)}(\bfx).$
Since
\begin{equation}\label{support_fct_int}
K = \bigcap_{\bfx \in \mathbb R^5 \setminus \{0\}} \, \{\bfy\in\mathbb R^5 \mid \langle \bfx,\bfy\rangle \leq h_K(\bfx)\},
\end{equation}
this means $\varphi^{-1}(L)\subset K$.

For every point $\bfx\in\mathbb R^5$ we have $\mathrm{NJ}(\varphi,\bfx)=\rho_1\cdot \rho_2$. 
For a fixed $\boldsymbol\rho\in\mathbb R^3$ the fiber $\varphi^{-1}(\boldsymbol\rho)$ consists of the product of two circles (all points $\bfx$ with $\sqrt{x_1^2+x_2^2}=\rho_1$ and $\sqrt{x_3^2+x_4^2}=\rho_1$) and two points ($-x_5$ and~$x_5$). Therefore, the fibers of $\varphi$ have volume $2 \mathrm{vol}(\S^1)^2 = 2(2\pi)^2$. Then, by the coarea formula \cref{coarea_formula},
\begin{equation}\label{eq8}\mathrm{vol}(K)\geq \mathrm{vol}(\varphi^{-1}(L)) = \int_{\mathbb R^5} \delta_{\varphi^{-1}(L)}(\bfx)\;\mathrm d\bfx = 2(2\pi)^2 \cdot \int_L \rho_1\cdot \rho_2\;\mathrm d\boldsymbol \rho,
\end{equation}
where $\delta_{\varphi^{-1}(L)}$ is the indicator function of the interior of $\varphi^{-1}(L)$.

We have $\mathbf{0}\in L$. Since $\langle \tfrac{2}{\pi^2}\bfe_1, \boldsymbol\rho\rangle = \tfrac{2}{\pi^2} \rho_1  \leq h_L(\boldsymbol\rho)$ for all $\boldsymbol\rho\neq \mathbf{0}$, we also have, by \cref{support_fct_int}, 
$$\bfp_1 := \tfrac{2}{\pi^2}\,\bfe_1\in L\quad \text{ and, similarly, }\quad \bfp_2:=\tfrac{2}{\pi^2}\,\bfe_2\in L, \quad \bfp_3:=\tfrac{1}{\pi}\,\bfe_3\in L.$$
Using \texttt{Mathematica} \cite{Mathematica} we prove that $$\lambda_1(\bfp_i+\bfp_3),\; \lambda_2(\bfp_i+\tfrac{2}{3}\bfp_3), \; \lambda_3(\tfrac{2}{3}\bfp_i+\bfp_3),\; \lambda_4(\bfp_i+\tfrac{1}{3}\bfp_3),\; \lambda_5(\tfrac{1}{3}\bfp_i+\bfp_3)\in L, \quad i=1,2,$$ 
where $\lambda_1=0.73, \lambda_2 = 0.86, \lambda_3 = 0.85, \lambda_4=0.966, \lambda_5 = 0.957.$ 

By convexity,~$L$ contains the convex hull of all these points. We define 
\begin{align}\label{def_P}
P\ :=\ & \operatorname{conv} \big(\{\mathbf 0, \,\bfe_1,\ \bfe_2,\ \bfe_3,\ \lambda_1(\bfe_1+\bfe_3)\} \\
&\cup \{\lambda_2(\bfe_i+\tfrac{2}{3}\bfe_3),\ \lambda_3(\tfrac{2}{3}\bfe_i+\bfe_3),\ \lambda_4(\bfe_i+\tfrac{1}{3}\bfe_3),\ \lambda_5(\tfrac{1}{3}\bfe_i+\bfe_3)\mid i=1,2\}\big)\nonumber
\end{align}
(see \cref{fig2}). 
Then, using the coarea formula \cref{coarea_formula} we have 
\begin{equation}\label{eq10}
\int_L\rho_1\cdot \rho_2\;\mathrm d\boldsymbol \rho\geq \left(\frac{2}{\pi^2}\right)^4 \cdot \frac{1}{\pi}\cdot \int_P\rho_1\cdot \rho_2\;\mathrm d\boldsymbol \rho.
\end{equation}
We evaluate this integral using \texttt{Mathematica} \cite{Mathematica}  and get $\int_P\rho_1\cdot \rho_2\;\mathrm d\boldsymbol\rho \geq 0.0236165$. 

\begin{proof}[Proof of \cref{thm: lower bound}]
By \cref{eq9}, we have $\mean_{L\sim \psi} \# (\mathcal E\cap L)  = 5!\cdot \frac{\pi^3}{4}\cdot \mathrm{vol}(K)$. Above we have shown 
$$\mathrm{vol}(K)\,\stackrel{\text{\cref{eq8}}}{\geq }\, 2(2\pi)^2 \cdot \int_L \rho_1\cdot \rho_2\;\mathrm d\boldsymbol \rho \, \stackrel{\text{\cref{eq10}}}{\geq}  \, \frac{2^7}{\pi^7}\cdot \int_P\rho_1\cdot \rho_2\;\mathrm d\boldsymbol \rho\geq \frac{2^7}{\pi^7}\cdot 0.0236165.$$ 
So, 
$\mean_{L\sim \psi} \# (\mathcal E\cap L)\geq 5! \cdot \frac{2^5}{\pi^4} \cdot 0.0236165  \geq 0.93$.
\end{proof}

The implementations of all numerical computations made in this contribution can be found:

\url{https://mathrepo.mis.mpg.de/average_degree/index.html}

\bigskip

\bibliographystyle{plain}
\bibliography{literatur.bib}

\end{document}